\documentclass{amsart}

\title[Magnitudes of Compact Sets]{On the Magnitudes of Compact Sets in
  Euclidean Spaces }
\author{Juan Antonio Barcel\'o and Anthony Carbery}
\thanks{The first author was supported by Spanish Grant MTM2011-28198.}

\address{Juan Antonio Barcel\'o,
ETSI de Caminos, Canales y Puertos,  
Universidad Polit\'ecnica de Madrid, 
28040 Madrid, Spain } 
\email{juanantonio.barcelo@upm.es}

\address{Anthony Carbery, 
School of Mathematics and Maxwell Institute for Mathematical Sciences, 
University of Edinburgh,
JCMB, 
Peter Guthrie Tait Road,
King's Buildings, 
Mayfield Road, 
Edinburgh, EH9 3FD, 
Scotland.} 
\email{A.Carbery@ed.ac.uk}

\date{9th July 2015, revised 13th July 2016}
\usepackage{amssymb}
\usepackage{amsmath}
\usepackage{url}

\newtheorem{theorem}{Theorem}
\newtheorem{proposition}{Proposition}
\newtheorem{lemma}{Lemma}
\newtheorem{corollary}{Corollary}

\newtheorem{conjecture}{Conjecture}

\begin{document}

\begin{abstract}
The notion of the magnitude of a metric space was introduced by Leinster in \cite{Lein13} and 
developed in \cite{Meckes13}, \cite{LeWi}, \cite{Meckes14}, \cite{Will2} and \cite{LM16}, but 
the magnitudes of familiar sets in Euclidean space are only 
understood in relatively few cases. In this paper we study the
magnitudes of compact sets in Euclidean spaces. We first describe the
asymptotics of the magnitude of such sets in both the small and large-scale regimes.
We then consider the magnitudes of compact convex sets with 
nonempty interior in Euclidean spaces of odd dimension, and relate them to the boundary behaviour 
of solutions to certain naturally associated higher order elliptic boundary value problems 
in exterior domains. We carry out calculations leading to an algorithm for explicit evaluation 
of the magnitudes of balls, and this establishes the convex magnitude conjecture of Leinster 
and Willerton \cite{LeWi} in the special case of balls in dimension
three. In general we show that the magnitude of an odd-dimensional
ball is a rational function of its radius. In addition to
Fourier-analytic and PDE
techniques, the arguments 
also involve some combinatorial considerations.
\end{abstract}

\maketitle
\section{Introduction}\label{Intro}
\noindent
Motivated by considerations of a category-theoretic nature, Leinster \cite{Lein13} has 
introduced the notion of the {\em magnitude} of a metric space. Magnitude is an important new 
numerical invariant 
of a metric space which shares some of the more abstract properties of the Euler characteristic 
of a category (or of a topological space), and indeed both can be seen as special cases of the 
notion of the Euler characteristic or magnitude of an {\em enriched category}. In particular, the inclusion-exclusion 
principle enjoyed by the Euler characteristic provides important motivation for the hoped-for
properties of magnitude. More generally, magnitude is designed to capture the 
``essential size'' of a metric space in a more subtle way than cruder measures such as 
cardinality or diameter, and at the same time it will also contain further significant 
geometric information concerning the space. For a much more detailed discussion of these 
issues see \cite{Lein13}, \cite{LeWi}, \cite{Will2} and \cite{LM16}.

\medskip
\noindent
Leinster's definition of the magnitude of a finite metric space bears close resemblance to notions 
of a potential-theoretic nature, and Meckes \cite{Meckes13} and \cite{Meckes14} has developed this 
perspective to the point where a tractable definition of the magnitude of a positive-definite 
compact metric space can now be given in terms analogous to those of classical capacity. This provides 
the starting point for our investigations. 

\medskip
\noindent
Before describing our results, we give a little more informal background on magnitude in order that our 
contributions can be placed in context.

\medskip
\noindent
\subsection{Definitions of magnitude and connection with capacity.}
Given a finite metric space $(X,d)$, Leinster \cite{Lein13} defined its magnitude as the value 
$$ |X| = \sum_{x \in X} w(x)$$ 
whenever $w : X \to \mathbb{R}$ satisfies 
$$\sum_{y \in X} e^{-d(x,y)} w(y) = 1$$
for all $x \in X$. 

\medskip
\noindent
It is easy to check that any two such $w$ will give the same value for the magnitude, and
if no such $w$ exists we declare the magnitude to be undefined.
Under the mild additional condition that $X$ be positive-definite (meaning 
that the matrix $(e^{-d(x,y)})_{x, y \in X}$ is positive-definite), its magnitude 
is defined, see \cite{Lein13}, \cite{Meckes13} and \cite{Meckes14}. 
Finite subsets of Euclidean spaces are always positive-definite, \cite{Lein13}. It is easy to check 
(see below for the argument) that $|\emptyset| = 0$, $|\{x\}| = 1$ and that if $X_N$ is 
an $N$-point space consisting of the vertices of a simplex in $\mathbb{R}^{n}$ (with $n \geq N-1$) which 
are all equidistant $t$ from one another, then 
\begin{equation}\label{finite}
|X_N| = \frac{N}{1+(N-1)e^{-t}}.
\end{equation}

\medskip
\noindent
We describe a compact metric space as positive-definite if every finite subset is
positive-definite, and the magnitude of a compact positive-definite metric
space $(X,d)$ has been defined \cite{Meckes13} as
$$|X| = \sup \{|\Xi| \, : \, \Xi \mbox{ a finite subset of } X\}.$$
(Once again, any compact subset of a Euclidean space is positive-definite.)

\medskip
\noindent
Bearing in mind Leinster's definition, it is natural to consider the class of finite signed Borel 
measures $\mu$ on a compact positive-definite metric space $X$, and for such a $\mu$ introduce its 
{\em potential function} $\Phi_\mu$ given by
$$ \Phi_\mu(x) := \int_X e^{-d(x,y)} {\rm d} \mu(y).$$
If there is such a finite signed Borel measure $\mu$ on $X$ satisfying
$$ \Phi_\mu(x)= \int_X e^{-d(x,y)} {\rm d} \mu(y) \equiv 1 \mbox{ on }X,$$ 
then $\mu$ is called a {\em weight measure} for $X$. 
It is known (\cite{Meckes13}, \cite{Meckes14}) that if $X$ is compact and positive-definite, then 
$|X| = \mu(X)$ for any weight measure $\mu$.

\medskip
\noindent
If we take $X = [-R,R] \subseteq \mathbb{R}$ with the usual 
metric, then one simply checks using integration by parts that, with ${\rm d}x$ denoting Lebesgue measure
on $\mathbb{R}$, $\frac{1}{2}(\delta_{-R} + \delta_{R} + {\rm d}x|_{[-R,R]})$ is a 
weight measure for $X$, and hence $|[-R,R]| = R + 1$. (See \cite{Lein13}
and \cite{Meckes14}.) This is the only example of a nontrivial compact 
convex set in a Euclidean space whose magnitude was hitherto known. 

\medskip
\noindent
In nearly all examples for which the magnitude
is known explicitly one can fairly easily identify a weight measure; in particular when the metric space $X$ enjoys 
a lot of symmetry one expects a weight measure to reflect this symmetry and this leads to a limited range 
of possibilities. For example, in the case of the simplex mentioned above, we just set $w(y) = a$ for all $y$ and observe that 
the equation $\sum_{y \in X} e^{-d(x,y)} w(y) = 1$ becomes $a + (N-1)e^{-t}a = 1$, leading to \eqref{finite}. 
For more examples see \cite{Lein13}, \cite{LeWi}, \cite{Will2}. Our approach here is different and is 
motivated by connections with differential equations.

\medskip
\noindent
There is a clear analogy between magnitude and the notion of capacity as developed in classical 
potential theory, and which we now describe in a very informal manner. 
A possible definition of the $\alpha$-capacity of a compact metric space is
$$ \mbox{cap}_\alpha(X) =  \sup \mu(X)$$
where the sup is taken over all positive finite Borel measures $\mu$ on $X$ such that
$$ \int_X d(x,y)^{- \alpha} {\rm d} \mu(y) \leq 1 \mbox{ on }X.$$ 
The study of the potentials $ \int_X d(x,y)^{- \alpha} {\rm d} \mu(y)$ and the associated 
$\alpha$-capacities has a long and distinguished history. When we are in Euclidean space 
$\mathbb{R}^n$ with the usual metric, we can take advantage of the Fourier transform 
to characterise the $\alpha$-capacity of a compact subset $X$ for $0 < \alpha < n$ as
$$\mbox{cap}_\alpha(X) = c_{n,\alpha}\inf\left\{ \int_{\mathbb{R}^n}|\xi|^{n-\alpha}|\widehat{f}(\xi)|^2 {\rm d}\xi 
\, : \, f \geq 1 \mbox{ on } X\right\}$$
where $c_{n,\alpha}$ is a certain dimensional constant and $\, \widehat{}\, $ denotes the Fourier transform. See for example \cite{AH} for a thorough and detailed discussion of potential theory and capacity in the Euclidean setting.


\medskip
\noindent
Meckes in \cite{Meckes14} develops the analogy between magnitude and capacity beyond the formal level, and gives an extremal characterisation of 
magnitude in some generality. In the case of compact sets $X$ in Euclidean space $\mathbb{R}^n$ this 
characterisation can be realised as
\begin{equation}\label{def} 
|X| = \frac{1}{n! \omega_n} \inf\left\{\|f\|^2_{H^{(n+1)/2}(\mathbb{R}^n)} \, : \, 
f \in H^{(n+1)/2}(\mathbb{R}^n), f \equiv 1 \mbox{ on } X\right\}
\end{equation}
where $\omega_n$ is the volume of the unit ball in $\mathbb{R}^n$ and $H^{m}(\mathbb{R}^n)$
is the Sobolev space of functions whose derivatives of order up to $m$
are in $L^2(\mathbb{R}^n)$. More precisely,
$H^{m}(\mathbb{R}^n)$ is the space of Bessel potentials of order $m$,
and its norm is given by 
$$  \|f \|^2_{H^{m}(\mathbb{R}^n)} := \|(I - \Delta)^{m/2} f
\|^2_{L^2(\mathbb{R}^n)}.$$
(There are many different -- but equivalent -- norms on the space $H^{m}(\mathbb{R}^n)$
which are regularly employed in the literature, especially when $m$
is an integer. We emphasise that throughout this paper, we exclusively
use the definition above.) Note that when $m = (n+1)/2 > n/2$,
functions in $H^m$ are continuous by the Sobolev embedding theorem so 
the prescription $ f \equiv 1 \mbox{ on } X$ makes sense pointwise.

\medskip
\noindent
One can informally motivate this result of Meckes
in line with classical potential theory as follows. Adopting the standard convention from Fourier analysis -- 
which differs from that used by Meckes in \cite{Meckes13} and \cite{Meckes14} -- 
that the Fourier transform is given by
$$ \widehat{f}(\xi) = \int_{\mathbb{R}^n} f(x) e^{- 2 \pi i x \cdot \xi} {\rm d}x,$$
then it defines an isometry on $L^2$, and we have 
\begin{equation}\label{FT}
\widehat{e^{-|\cdot|}}(\xi) =  n! \omega_n \left(1 + 4 \pi^2
|\xi|^2\right)^{-\frac{n+1}{2}}.
\end{equation}
(See for example Stein, \cite{Stein}.) So if $\mu$ is a weight measure 
for $X \subseteq \mathbb{R}^n$ we can extend it to be a measure on the whole of $\mathbb{R}^n$ in the canonical way,
$\Phi_\mu$ becomes defined on all of $\mathbb{R}^n$ and we have $\Phi_\mu(x) = 1$ on $X$. 
On the other hand, by taking Fourier transforms, 
$$\widehat{\Phi_\mu}(\xi) =  n! \omega_n \left(1 + 4 \pi^2 |\xi|^2\right)^{-\frac{n+1}{2}} \widehat{\mu}(\xi).$$
Therefore
\begin{equation*}
\begin{aligned}
\mu(X) &= \int_X \Phi_\mu(x) {\rm d}\mu(x) = \int_{\mathbb{R}^n} \int_{\mathbb{R}^n} e^{-|x-y|} {\rm d}\mu(y) {\rm d}\mu(x) \\
&= \int_{\mathbb{R}^n}  n! \omega_n \left(1 + 4 \pi^2 |\xi|^2\right)^{-\frac{n+1}{2}} |\widehat{\mu}(\xi)|^2 {\rm d} \xi \\
&= \frac{1}{n! \omega_n} \int_{\mathbb{R}^n} |\widehat{\Phi_\mu}(\xi)|^2 \left(1 + 4 \pi^2 |\xi|^2\right)^{\frac{n+1}{2}} {\rm d} \xi
\end{aligned}
\end{equation*}
and this last expression is exactly $\frac{1}{n! \omega_n}
\|\Phi_\mu\|^2_{H^{(n+1)/2}(\mathbb{R}^n)}$.

\medskip
\noindent 
It is easy to see using Plancherel's theorem and the binomial theorem that when $m$ is an integer the inner product 
$\langle \cdot, \cdot \rangle_{H^m(\mathbb{R}^n)}$ corresponding to the $H^m$ norm is given by
\begin{equation}\label{xsd1}
\langle f,g \rangle_{H^m(\mathbb{R}^n)} := \sum_{j=0}^m {{m}\choose{j}} \int_{\mathbb{R}^n} D^j f \cdot \overline{D^{j} g}
\end{equation}
where 
$D^j f = \Delta^{j/2} f$ for $j$ even and $D^j f = \nabla \Delta^{(j-1)/2} f$ for
$j$ odd. We shall make systematic use of this in what follows.

\medskip
\noindent
One should note that although the right-hand side of equation \eqref{def} apparently depends upon 
the Euclidean space in which $X$ sits, the original definition of the magnitude of $X$ is an intrinsic 
metric invariant, and so formula \eqref{def} will hold for whatever Euclidean space $\mathbb{R}^n$ in 
which we can embed $X$.  

\medskip
\noindent
\subsection{Magnitude and geometric invariants.}
It is perhaps not {\em a priori} clear why, despite its ubiquity, the exponential function $s \mapsto e^{-s}$
is singled out to appear in the definition of magnitude in composition with the metric. Ultimately the reason for this is 
that its key property of converting addition to multiplication distinguishes it as the (essentially) unique
function $\Psi$ such that, if $\Psi(d(x,y))$ is the kernel of the potential function, then the 
magnitude of a finite metric space (as defined via $\Psi$) and the Euler characteristic of a 
finite category share a common generalisation in the Euler characteristic or magnitude of an enriched category. 
In particular, with $\Psi(s) = e^{-s}$, the triangle inequality gives us 
$$ \Psi(d(x,y)) \geq \Psi(d(x,z))\Psi(d(z,y))$$
for all $x,y,z \in X$; this inequality is to be thought of as analogous to 
the category-theoretic composition of a map from $x$ to $z$ with a map from $z$ to $y$ to obtain a map from $x$ to 
$y$. For more details see \cite{Lein13}.

\medskip
\noindent
In the category of finite sets, magnitude is just cardinality, which obviously satisfies the familiar 
inclusion-exclusion principle 
$$ \# (A \cup B) + \# (A \cap B) = \# (A) + \# (B).$$
Similarly, in the category of finite simplicial complexes, the classical Euler characteristic $\chi$ can be characterised
by $\chi(\emptyset) = 0$, $\chi(\{x\}) = 1$ and
$$ \chi (A \cup B) + \chi (A \cap B) = \chi (A) + \chi (B).$$
However, magnitude in the setting of finite metric spaces is manifestly not finitely additive as the example
of the simplex with equidistant vertices, \eqref{finite}, indicates. Nevertheless, formula \eqref{finite} tells
us that as $t \to \infty$, the magnitude of $X_N$ does asymptotically satisfy the inclusion-exclusion 
principle since $|X_N|$ approaches $N$ as $ t \to \infty$. More generally if $(X,d)$ is any 
finite metric space and if we let $tX = (X, td)$ for $t > 0$, it is
known (see \cite{Lein13}) that $|tX| \to \#X$ as 
$t \to \infty$, and so inclusion-exclusion holds in an {\em asymptotic} sense for the class of finite metric spaces. 
On the other hand, 
if we consider the (admittedly somewhat limited) class $\mathcal{K}$ of compact convex sets in $\mathbb{R}$, 
then if $A, B \in \mathcal{K}$ are such that $A \cup B \in \mathcal{K}$, we have the {\em exact} 
inclusion-exclusion identity
$$ |A \cup B| + |A \cap B| = |A| + |B|$$
for magnitude (since the compact interval $[a,b]$ has magnitude $1 + (b-a)/2$ as we have seen above). 
These considerations lead one to hope that magnitude in general will satisfy the inclusion-exclusion 
principle in some asymptotic sense, and perhaps exactly so under certain more restricted circumstances 
-- such as in the presence of convexity. (One may remark that the classical theory of capacity of compact 
{\em convex} sets is richer than the general theory, and from the category-theoretic point of view convex 
sets are natural in so far as the path of minimal distance between any two  points is contained in the set.)

\medskip
\noindent
Examples of metric spaces in which there is a lot of symmetry, and therefore where it is possible to evaluate 
the magnitude directly, yield further insight into some of the geometric characteristics which magnitude 
might capture. In analogy with the calculation for the simplex above, if $X$ 
is a compact metric space and $G$ is a compact group with Haar measure $\nu$, acting 
transitively and isometrically on $X$, then the only 
sensible candidate for a weight measure for $tX = (X, td)$ should have {\em constant} density 
$\lambda a(t)$ (for some $\lambda > 0$) where 
$ a(t) \int_G e^{-td(x, gy)} {\rm d} \nu(g) = 1$ (the integral is independent of $x$ and $y$).
So the magnitude of $tX$ should be
$$ |tX| = \frac{\nu(G)}{\int_G e^{-td(x, gy)} {\rm d} \nu(g)}.$$
Exact evaluation or asymptotic analysis of such expressions as $t \to \infty$ has led to the following results:

\medskip
\noindent
{\bf Theorem A} (Willerton, \cite{Will2}.) {\em Let $X$ be a two-dimensional homogeneous compact manifold. Then
$$ |tX| = \frac{1}{2 \pi} {\rm Area}(X) t^2 + \chi(X) + O(t^{-2})$$
as $t \to \infty$, where $\chi$ is the classical Euler characteristic.} 

\medskip
\noindent
{\bf Theorem B} (Willerton, \cite{Will2}.) {\em Let $n \geq 1$. Then there is an explicit formula 
for $|t \mathbb{S}^n|$ with its geodesic metric, which yields
$$ |t\mathbb{S}^n| = \frac{1}{n! \omega_n} {\rm Vol}_n(\mathbb{S}^n) t^n + 
\frac{n+1}{6 n! \omega_n} {\rm tsc}(\mathbb{S}^n) t^{n-2} + \cdots + \chi(\mathbb{S}^n) + O(e^{-t})$$
as $t \to \infty$, where ${\rm Vol}_n$ denotes $n$-dimensional volume, $\chi$ is the classical Euler 
characteristic and ${\rm tsc}$ is the total scalar curvature.}

\smallskip
\noindent
(Note that $\chi(\mathbb{S}^n)$ is $2$ for $n$ even and $0$ for $n$ odd.) Willerton also has formulae for the 
magnitude of spheres in Euclidean spaces with the subspace metric which are explicit when $n=1$ and $2$ 
and give the first few terms in the asymptotic expansion as $t \to \infty$ for larger values of $n$. 
In addition, he has obtained the first four terms in the asymptotic expansion for a general $n$-dimensional 
homogeneous compact manifold.

\medskip
\noindent
Together with the result that for a nonempty compact interval $I \subseteq \mathbb{R}$ we have
$$ |tI| = \frac{1}{2} {\rm Length}(I) t + \chi(I)$$
(where $\chi(I)$, the Euler characteristic of $I$, equals $1$), these results of Willerton suggest that 
$ t \mapsto |tX|$ should have a polynomial flavour as $t \to \infty$, with the coefficients of the 
polynomial part encapsulating important geometric information concerning $X$. Combined with the results of
numerical calculations \cite{Will1} and some parallel results when the $l^1$-metric is used instead of the 
usual metric (see \cite{Lein13}), they provide evidence for the following conjectures of Leinster and 
Willerton from \cite{LeWi}: 

\begin{conjecture}[Leinster--Willerton]\label{LW1}
For suitable (i.e. so that perimeter and classical Euler characteristic are at least defined) 
compact $X \subseteq \mathbb{R}^2$ we have, as $t \to \infty$,
$$ |tX| =  \frac{1}{2 \pi} {\rm Area}(X) t^2 + \frac{1}{4}{\rm Perim}(X) t+ \chi(X) + o(1).$$
\end{conjecture}

\medskip
\noindent
Here we see that intermediate-dimensional geometrical characteristics (in this case the
perimeter) of $X$ are expected to play a role. We shall not be directly concerned with Conjecture \ref{LW1} 
in the current paper, but will instead focus on the following, known as the convex magnitude conjecture of 
Leinster and Willerton, \cite{LeWi}:
\begin{conjecture}[Leinster--Willerton]\label{LW2}
Suppose $X \subseteq \mathbb{R}^n$ is compact and convex. Then
$t \mapsto |tX|$ is a polynomial of degree $n$ and moreover
\begin{equation}\label{LWform}
 |tX| =  \frac{{\rm Vol}(X)}{n! \omega_n}t^n + \frac{{\rm Surf}(\partial X)}{2(n-1)!\omega_{n-1}} 
t^{n-1}+ \dots  + 1
= \sum_{i=0}^n \frac{1}{i! \omega_i} V_i(X) t^i
\end{equation}
where $V_i(X)$ is the $i$'th intrinsic volume of $X$ and $\omega_i$ is the volume of the unit 
ball in $\mathbb{R}^i$.
\end{conjecture}

\medskip
\noindent
The intrinsic volumes $V_i$ are classical integral-geometric invariants defined on the class of 
compact convex sets $\mathcal{K}$ in Euclidean space, and $V_i(X)$ encapsulates the quantitative 
$i$-dimensional information concerning $X$. In particular we have $V_i(tX) = t^i V_i(X)$.
Indeed, $V_n$ is ordinary volume, $V_{n-1}$ is half the surface area, and $V_i$ for $1 \leq i \leq n-2$ 
captures the $i$-dimensional information on the boundary $\partial X$, while $V_0(X) =1$ for
nonempty $X$ and $V_0(X) =0$ when $X$ is empty. Each $V_i$ is a {\em valuation} on
$\mathcal{K}$, which means that $V_i(\emptyset) = 0$, and if $A, B \in
\mathcal{K}$ are such that $A \cup B \in \mathcal{K}$, then we have the 
inclusion-exclusion identity
$$ V_i(A \cup B) + V_i(A \cap B) = V_i(A) + V_i(B).$$
For more detail see \cite{KR}. So the convex magnitude conjecture implies in particular that 
magnitude is a valuation on $\mathcal{K}$. On the other hand, it is
known by Hadwiger's theorem (see again \cite{KR}) 
that any valuation on $\mathcal{K}$ which is continuous with respect to the Hausdorff metric 
and which is invariant under Euclidean motions must be given by a
linear combination of the intrinsic volumes. It is also known (see \cite{McM1}, \cite{McM2})
that a valuation which is translation-invariant and monotone is necessarily continuous.
Clearly magnitude is invariant under Euclidean motions and is monotone, so 
the convex magnitude conjecture is true if and only if magnitude is a valuation {\em and} the conjecture 
is true for a suitable family of compact convex sets -- such as the closed balls -- which serves to 
normalise the precise coefficients of the polynomial arising.
It is therefore of considerable interest to determine the validity of 
the conjecture in the special case of closed balls in $\mathbb{R}^n$, and this issue 
is the main focus of the present paper. 

\medskip
\noindent
On the other hand, one does not see the inclusion-exclusion principle arising in classical potential 
theory, so perhaps a certain amount of caution is in order in approaching this issue.

\medskip
\noindent
To fix ideas, the convex magnitude conjecture predicts that the magnitude of 
the closed ball $B_R$ of radius $R$ in $\mathbb{R}^n$ will be: 
\begin{eqnarray*}
\begin{aligned}
& n=1 :  \; \;  R + 1 \\
& n=2 :  \; \;  \frac{R^2}{2!} + \frac{\pi R}{2} + 1\\
& n=3 :  \; \;  \frac{R^3}{3!} + R^2 + 2R + 1\\
& n=4 :  \; \;  \frac{R^4}{4!} + \frac{ \pi R^3}{8} + \frac{3R^2}{2} + \frac{3 \pi R}{4} +1\\
& n=5 :  \; \;  \frac{R^5}{5!} + \frac{R^4}{9} + \frac{2R^3}{3} + 2R^2 + \frac{8R}{3} +1\\
& n=6 :  \; \;  \frac{R^6}{6!} + \frac{\pi R^5}{128} + \frac{5R^4}{24} + \frac{5 \pi R^3}{16}
+  \frac{5R^2}{2} + \frac{15 \pi R}{16} +1\\
& n=7 :   \; \; \frac{R^7}{7!} + \frac{R^6}{225} +  \frac{R^5}{20} + \frac{R^4}{3} + \frac{4R^3}{3}
+  3R^2 + \frac{16R}{5} +1,\\
\end{aligned}
\end{eqnarray*}
etc.

\medskip
\noindent
\subsubsection{What is currently known about the convex magnitude conjecture.}
It is clear from \eqref{def} that $t \mapsto |tX|$ is an increasing function of $t$ when $X$ is compact and 
convex. Meckes \cite{Meckes14} has shown that this function is continuous on $(0,\infty)$, and Leinster \cite{Lein13} 
(see also Meckes \cite{Meckes14}) has shown 
that for {\em arbitrary} compact sets in $\mathbb{R}^n$ we have
$$|X| \geq \frac{{\rm Vol}(X)}{n! \omega_n}.$$
When $n=1$ the conjecture is true (see above), but when $n \geq 2$, methods based entirely on symmetry -- and 
indeed on weight measures -- will not suffice to resolve it, and instead we turn to 
techniques of differential equations.

We remark that a natural analogue of the convex magnitude conjecture does hold for convex bodies 
in $\mathbb{R}^n$ when instead of the $ell^2$ norm we use the $\ell_1^n$ norm. See Theorem 4.6(2) of 
\cite{LM16}.

\medskip
\noindent
\subsection{Main results of the present paper.}
We shall first establish that the asymptotics as $t \to 0$ and $t \to
\infty$ which are predicted by the Leinster--Willerton convex magnitude
conjecture do indeed hold, and in fact do so more generally for
nonempty compact sets. That is, we prove:

\begin{theorem}\label{asymprat}
Let $X$ be a nonempty compact set in $\mathbb{R}^n$. Then 
$$ |RX| \to 1 \mbox{ as }R \to 0$$ 
and
$$ R^{-n}|RX| \to  \frac{{\rm Vol}(X)}{n! \omega_n} \mbox{ as }R \to \infty.$$ 
\end{theorem}

\medskip
\noindent
As a consequence of the first statement we have that $t \mapsto |tX|$
is continuous also at $t=0$.

\medskip
\noindent
We shall next be concerned with the validity of the convex magnitude conjecture for the 
class of convex bodies in Euclidean space $\mathbb{R}^n$, that is, compact convex sets with nonempty 
interior. Our starting point will be Meckes' formula \eqref{def} for the magnitude of such a convex body 
and we shall try to relate this to the formula \eqref{LWform} appearing in the convex magnitude conjecture.

\medskip
\noindent
Meckes has already observed \cite{Meckes14} that the extremiser in \eqref{def} exists and is unique, and also satisfies
the associated Euler--Lagrange equation 
\begin{equation}\label{EL}
(I - \Delta)^{(n+1)/2} f = 0 \mbox{ weakly on } \mathbb{R}^n \setminus X
\end{equation}
which, when $n$ is odd, is an elliptic differential (as opposed to pseudodifferential) equation. 
{\em In order to be able to work with differential rather than pseudodifferential equations we shall 
from Section \ref{Sec_ODE} onwards assume that $n$ is odd.} We shall first show that there is a unique 
member of $H^{(n+1)/2}(\mathbb{R}^n)$ which is a solution to \eqref{EL} and which is identically $1$ on $X$.
This solution therefore extremises the energy $\|f\|^2_{H^{(n+1)/2}(\mathbb{R}^n)}$ over all  
$f \in H^{(n+1)/2}(\mathbb{R}^n)$ with $f \equiv 1$ on $X$.

\medskip
\noindent
The convex magnitude conjecture predicts a single term corresponding to the volume of $X$, with the
remaining terms mainly relating to the boundary $\partial X$. Our next task is therefore to develop a formula 
for the extremal energy expressed in terms of ${\rm Vol}(X)$ and the values of the extremising function 
near $\partial K$. This is carried out in order to highlight the contribution to the extremal energy arising 
from the boundary of $X$ and hence ultimately to facilitate a comparison between \eqref{def} 
and \eqref{LWform}. The formula is analogous to classical representations of capacities as 
boundary integrals of functions of solutions to the associated partial differential equations, and it 
appears in Theorem \ref{main} below.

\medskip
\noindent
We then turn to the problem of explicit identification of the extremiser in the case that $X$ 
is a nontrivial closed ball. This is achieved by using spherical symmetry to reduce matters to consideration 
of certain ordinary differential equations.

\medskip
\noindent
Finally, having obtained the extremiser, we use our formula from Theorem \ref{main} to obtain 
an explicit evaluation of the magnitude of a closed ball. (Strictly speaking we could in principle avoid the use of 
our formula at this stage and instead simply calculate $\|f\|^2_{H^{(n+1)/2}(\mathbb{R}^n)}$ directly for the 
extremiser, but we wish to try to identify precisely the contributions of the boundary of the ball 
to its magnitude, and in any case application of the formula is fairly
direct.) Leinster and Meckes \cite{LM16} have recently developed an alternate approach to calculation of 
$\|f\|^2_{H^{(n+1)/2}(\mathbb{R}^n)}$ which bypasses our Theorem \ref{main}, but which does not emphasise the role 
of the boundary. See the remark at the end of Section \ref{sec_formula}. 

\medskip
\noindent
In fact, the explicit identification of the extremiser is not an entirely straightforward process. Our 
procedures amount to giving an algorithm for its identification in any odd dimension, and thus yield 
an algorithm for the formula for the magnitude of a closed ball in any odd dimension. The formulae
become more complex as dimension increases and so we provide explicit versions only in dimensions $1$, $3$, $5$ 
and $7$.

\medskip
\noindent
The first upshot of these calculations is that the Leinster--Willerton convex magnitude conjecture 
is true for closed balls in three dimensions:

\begin{theorem}\label{zxc}
The magnitude of the closed ball of radius $R>0$ in $\mathbb{R}^3$ is
$$\frac{R^3}{3!} + R^2 + 2R + 1.$$
\end{theorem}

\medskip
\noindent
However, in higher odd dimensions the magnitude is not a polynomial in the radius, thus disproving the 
convex magnitude conjecture in general. In dimension five we have:

\begin{theorem}\label{cxz}
The magnitude of the closed ball of radius $R>0$ in $\mathbb{R}^5$ is
$$ \frac{R^5}{5!} +  \frac{3 R^5 + 27 R^4 + 105 R^3 + 216 R^2 + 72}{24(R +3)}.$$
\end{theorem}

\medskip
\noindent
Note that when $R = 0$ this expression takes the value $1$, and as $R \to \infty$ the leading term is
$R^5/5!$, but the formula does not agree with the conjectured value of
$$\frac{R^5}{5!} + \frac{R^4}{9} + \frac{2R^3}{3} + 2R^2 + \frac{8R}{3} +1.$$
If one expands the expression for the magnitude asymptotically for $R > 3$ one obtains 
$$ \frac{R^5}{5!} + \frac{R^4}{8} + \frac{3R^3}{4} + \frac{17R^2}{8} +  \frac{21R}{8} + \frac{9}{8} +O(\frac{1}{R})$$
as $R \to \infty$. (Notice that some, but not all, of the coefficients
here differ by a multiplicative factor of $9/8$'ths from the conjectured values.) The role of the value $R = -3$ 
remains mysterious.

\medskip
\noindent
Although the magnitude is not a polynomial for odd $n \geq 5$, it is the next best thing:

\begin{theorem}\label{rat}
The magnitude of the closed ball of radius $R>0$, in a Euclidean space $\mathbb{R}^n$ of odd 
dimension, is a rational function of $R$ with rational coefficients. 
\end{theorem}

\medskip
\noindent
It seems possible that one might be able to take the coefficients in the polynomials 
featuring in the rational function to be nonnegative, but this has not been verified in general. 
One can show that the denominator can be taken to be a polynomial of degree strictly less than 
$\frac{3n^2 - 2n +7}{8}$.

\medskip
\noindent
There remain the tantalising questions of whether magnitude is a valuation in dimension three -- and indeed 
whether the convex magnitude conjecture is true in general in dimension three -- and of what happens in even 
dimensions, especially dimensions two and four. Is $|RX|$ a rational function of $R$ 
for general compact convex bodies $X$ in odd dimensions? If so, do the coefficients of the polynomials 
have a geometric significance? (See Section \ref{sec_final} for a brief 
discussion of cuboids and ellipsoids in three-dimensional space.)

\medskip
\noindent
\subsection{Structure of the paper.}
The paper falls into three main parts. 

\medskip
\noindent
In the first part, in Section \ref{sec:asymp}, we begin by giving the asymptotic results of 
Theorem \ref{asymprat}, based on little more than elementary Fourier analysis. 

\medskip
\noindent
In the second part we develop the general PDE theory for our problem.
In Section \ref{sec_Sob} we set out the basic facts about Sobolev spaces 
which we shall need. In Section \ref{sec_var}
we formulate the variational problem for whose solution we develop a formula in Section \ref{sec_formula}.

\medskip
\noindent
In the third part we focus on the spherically symmetric situation. In Section \ref{combsec} we give some 
combinatorial preliminaries. In Section \ref{Sec_ODE} we find the general solution of the PDE problem 
from Section \ref{sec_var} in the spherically symmetric case. In Section \ref{sec:alg} we fit the 
boundary conditions and develop an 
algorithm to identify our explicit solution and its magnitude. The discussion here leads to the proof 
of Theorem \ref{rat} in Subsection \ref{sec:pf}. In Section \ref{sec:imp} we implement 
the algorithm in dimensions $1$, $3$, $5$ and $7$, leading to Theorems \ref{zxc} and \ref{cxz}. 

\medskip
\noindent
Finally, in Section \ref{sec_final} we make some concluding remarks.

\medskip
\noindent
\subsection{Acknowledgements.}
\noindent
The authors would like to thank Tom Leinster for introducing them to
the fascinating and beautiful subject of magnitude, and for patiently and carefully explaining its
category-theoretic origins in a series of seminars and other discussions at the University
of Edinburgh. They would also like to thank Mark Meckes for a number of illuminating email exchanges and comments.
Finally, they would like to thank the referees whose careful reading of the manuscript has led to several 
clarifications and presentational improvements. 

\section{Asymptotic results}\label{sec:asymp}
\noindent
We prove the asymptotic statements of Theorem \ref{asymprat} which are based on elementary 
Fourier analysis.


\medskip
\noindent
Let us first consider the asymptotic behaviour as $R \to \infty$. 
If $f \in H^{(n+1)/2}(\mathbb{R}^n)$ satisfies $f \equiv 1$ on a compact set $X$, it must manifestly
satisfy $\|f\|^2_{H^{(n+1)/2}(\mathbb{R}^n)} \geq \|f\|^2_{L^2(\mathbb{R}^n)} \geq {\rm Vol}(X)$, so that 
\begin{equation}\label{wsx}
|X| \geq \frac{{\rm Vol}(X)}{n! \omega_n}
\end{equation}
as has been observed by Leinster \cite{Lein13} and Meckes \cite{Meckes14}. 
But also
\begin{equation*}
\begin{aligned}
\|f(R^{-1} \cdot)\|^2_{H^{(n+1)/2}(\mathbb{R}^n)} 
&= \int_{\mathbb{R}^n} |R^n \widehat{f}(R \xi)|^2 (1 + 4 \pi^2
|\xi|^2)^{(n+1)/2} {\rm d}\xi \\
&= R^n \int_{\mathbb{R}^n} |\widehat{f}(\xi)|^2 (1 + 4 \pi^2
(|\xi|/R)^2)^{(n+1)/2} {\rm d}\xi
\end{aligned}
\end{equation*}
so that, by the monotone (or dominated) convergence theorem, 
\begin{equation}\label{gtb}
R^{-n}\|f(R^{-1} \cdot)\|^2_{H^{(n+1)/2}(\mathbb{R}^n)} \to \int_{\mathbb{R}^n} |\widehat{f}(\xi)|^2 {\rm d}\xi =  
\int_{\mathbb{R}^n} |f(x)|^2 {\rm d} x
\end{equation}
as $R \to \infty$. So we have 
$$\frac{{\rm Vol}(X)}{n! \omega_n} \leq R^{-n}|RX| \leq \frac{R^{-n} \|f(R^{-1} \cdot)\|^2_{H^{(n+1)/2}(\mathbb{R}^n)}}{n! \omega_n}
\to \frac{\int_{\mathbb{R}^n} |f(x)|^2 {\rm d} x}{n! \omega_n}$$
using \eqref{wsx}, \eqref{def} and \eqref{gtb} successively. Now we can find 
$f \in  H^{(n+1)/2}(\mathbb{R}^n)$ with $f \equiv 1$ on $X$ such that $\|f\|_2^2$ is as close 
as we like to ${\rm Vol}(X)$. Indeed, with $\Phi : \mathbb{R}^n \to [0,\infty)$ a smooth function of compact support
in $\{|x| \leq 1\}$ with $\int \Phi = 1$, $\Phi_r(x) := r^{-n} \Phi( r^{-1}x)$ and 
$X_r := \{x \in \mathbb{R}^n \, : \, d(x,X) \leq r\}$ we have that $f_r := \Phi_r \ast \chi_{X_r}$ is a nonnegative 
smooth function of compact support which satisfies $f_r(x) = 1$ for $x \in X$, and $f_r(x) \to 0$ as $r \to 0$ 
for $x \notin X$. So by the dominated convergence theorem, $\int f_r^2 \to {\rm Vol}(X)$ as $r \to 0$. Therefore 
$$ R^{-n}|RX| \to  \frac{{\rm Vol}(X)}{n! \omega_n}$$
as $R \to \infty$. 

\medskip
\noindent
Now let us consider what happens as $R \to 0$. Note that we may assume that $n$ is odd since as we 
remarked above in Section \ref{Intro}, magnitude is intrinsically defined.

\medskip
\noindent
If $f \in H^{(n+1)/2}(\mathbb{R}^n)$ and $f(0) = 1$,
then 
$$ 1 = |f(0)| \leq \sup_x |f(x)| \leq \|\widehat{f}\|_{L^1(\mathbb{R}^n)}$$ 
$$\leq \left(\int |\widehat{f}(\xi)|^2(1 + 4 \pi^2 |\xi|^2)^{(n+1)/2} {\rm d}\xi \right)^{1/2} 
\left(\int \frac{{\rm d}\xi}{(1 + 4 \pi^2 |\xi|^2)^{(n+1)/2}} \right)^{1/2}$$ 
so that
$$ \| f \|^2_{H^{(n+1)/2}(\mathbb{R}^n)} \geq \left(\int \frac{{\rm d}\xi}{(1 + 4 \pi^2 |\xi|^2)^{(n+1)/2}} \right)^{-1}$$
with equality if and only if $\widehat{f}(\xi)$ is the scalar multiple of $(1 + 4 \pi^2 |\xi|^2)^{-(n+1)/2}$ with
$f(0) =1$, or, equivalently, $f(x) = e^{-|x|}$. By \eqref{FT} and
Fourier inversion, 
$$\int \frac{{\rm d}\xi}{(1 + 4 \pi^2 |\xi|^2)^{(n+1)/2}} =
\frac{1}{n! \omega_n} \int \widehat{e^{-|\cdot|}}(\xi) {\rm d} \xi =
\frac{1}{n! \omega_n}e^{-|0|} = \frac{1}{n! \omega_n}$$
so that 
\begin{equation}\label{qaz}
\| f \|^2_{H^{(n+1)/2}(\mathbb{R}^n)} \geq n! \omega_n
\end{equation}
(with equality if and only if $f(x) = f_0(x) := e^{-|x|}$). This implies that $|X| \geq 1$ 
so long as $X \neq \emptyset$. 

\medskip
\noindent
Now, for $0 < R \leq 1$, let $f_R: \mathbb{R}^n \to [0,1]$ be 
a smooth function which satisfies 
\begin{equation*}
f_R(x) =
\left\{   
\begin{array}{llll}
1, \; &|x| \leq R \vspace{0.25cm}\\
e^{R} e^{-|x|}, &|x| \geq R^{1/2}
\end{array}     \right.
\end{equation*}
and, for $R \leq |x| \leq R^{1/2}$ and $|\alpha| \geq 1$,
\begin{equation*}
\left|{\left(\frac{\partial}{\partial x}\right)}^{\alpha} f_R(x)\right| \lesssim R^{-(|\alpha|-1)/2}.
\end{equation*}
We then have, (see \eqref{xsd1}), 
$$\|f_R\|^2_{H^{(n+1)/2}(\mathbb{R}^n)} = \sum_{j=0}^{(n+1)/2}{\frac{n+1}{2} \choose j} \int_{\mathbb{R}^n} |D^j f_R(x)|^2 {\rm d}x $$
$$=  \sum_{j=0}^{(n+1)/2}{\frac{n+1}{2} \choose j} \int_{|x| \leq R^{1/2}} |D^j f_R(x)|^2 {\rm d}x 
+  e^{2R}\sum_{j=0}^{(n+1)/2}{\frac{n+1}{2} \choose j} \int_{|x| \geq R^{1/2}} |D^j f_0(x)|^2 {\rm d}x$$
$$ = I + II.$$

\medskip
\noindent
We can estimate $I$ by the term corresponding to $j = (n+1)/2$, that is, 
$$ I \lesssim R^{-((n+1)/2 -1)}R^{n/2}= R^{1/2},$$
and by the dominated convergence theorem, $II$ tends to $\|f_0\|^2_{H^{(n+1)/2}(\mathbb{R}^n)}$ as $R \to 0$. 
Therefore $$\|f_R\|^2_{H^{(n+1)/2}} \to \|f_0\|^2_{H^{(n+1)/2}} = n! \omega_n$$ 
as $R \to 0$.
This shows that $|B(0,R)| \to 1$ as $R \to 0$. 

\medskip
\noindent
Since for any nonempty compact set $X$ we have
(after suitable translation) $\{0\} \subseteq RX \subseteq B(0, RM)$ for some $M >0$, we immediately 
deduce that $|RX| \to 1$ as $R \to 0$. (We thank Mark Meckes for pointing out this last implication to us.)
 
\section{Preliminaries on Sobolev spaces}\label{sec_Sob}
\noindent
Let $\Omega$ be an arbitrary open set in $\mathbb{R}^n$. We need to consider 
some variants of the Sobolev spaces $H^m(\Omega)$ consisting of complex-valued functions 
whose weak derivatives of order up to and including $m \in \mathbb{N}$ belong to 
$L^2(\Omega)$. We shall assume the standard properties of $H^m(\mathbb{R}^n)$ and 
$H^m(\Omega)$ without specific mention. See \cite{Evans} or \cite{Stein} for more details.

\medskip
\noindent
Firstly, the space $H^m_0(\Omega)$ 
is the completion of $C^\infty_c(\Omega)$, the class of smooth functions of 
compact support inside $\Omega$, under the inner product 
\begin{equation}\label{xsd}
\langle f,g \rangle_{H^m(\Omega)} := \sum_{j=0}^m {{m}\choose{j}} \int_{\Omega} D^j f \cdot \overline{D^{j} g}
\end{equation}
where 
$D^j f = \Delta^{j/2} f$ for $j$ even and $D^j f = \nabla \Delta^{(j-1)/2} f$ for
$j$ odd. Here and throughout 
$$\nabla f := \left(\frac{\partial f}{\partial x_1}, \dots, \frac{\partial f}{\partial x_n} 
\right) \mbox{ and } \Delta f = \sum_{i=1}^n \frac{\partial^2 f}{\partial x_i^2}.$$
The inner product $\langle \cdot, \cdot \rangle_{H^m}$ gives rise to the 
norm $\| \cdot \|_{H^m}$ and when $\Omega = \mathbb{R}^n$ we can easily see using 
the Fourier transform that
$$ \|f\|_{H^m(\mathbb{R}^n)}^2 = \|(I - \Delta)^{m/2} f\|_{L^2(\mathbb{R}^n)}^2.$$ 

\medskip
\noindent
We shall be working from now on with real-valued functions, and so we shall suppress 
the complex conjugate occurring in \eqref{xsd}.

\medskip
\noindent
Let $K \subseteq \mathbb{R}^n$ be compact and convex with nonempty interior.
If $f \in H^m(\mathbb{R}^n)$ and 
$g \in H^m({\rm{ int } }\, K)$ we say that $f = g$ on $K$ if 
$\| f-g \|_{H^m({\rm{ int } }\, K)} = 0$. If $f = g$ on $K$ in this sense, then clearly $f = g$ 
almost everywhere on int $K$; conversely if $f \in H^m(\mathbb{R}^n)$, 
$g \in H^m({\rm{ int } }\, K)$ and $f = g$ almost everywhere on $ \rm{int} \, K$, 
then the weak derivatives of $f$ and $g$ of order up to and including $m$ on $\rm{int} 
 \, K$ coincide, and so $f=g$ on $K$. If $m > n/2$ the Sobolev embedding theorem implies that 
the functions in $H^m(\mathbb{R}^n)$ are continuous and hence make sense pointwise on $K$; so the statement
$f=g$ on $K$ can be interpreted pointwise in this case. 


\medskip
\noindent
With this in mind, for our second variant we define
$$\tilde{H}^m_0(\mathbb{R}^n \setminus K) := \{ f \in  H^m(\mathbb{R}^n) \; : \; f = 0 \mbox { on } \, K \}.$$
This is clearly a closed subspace of $H^m(\mathbb{R}^n)$, and there are natural 
isometric embeddings
$$ H^m_0(\mathbb{R}^n \setminus K) \hookrightarrow \,  \tilde{H}^m_0(\mathbb{R}^n \setminus K) \hookrightarrow \, H^m(\mathbb{R}^n).$$
Under our hypothesis that the interior of $K$ is nonempty, the first embedding is surjective:


\begin{lemma}\label{Sob}
If $K$ has nonempty interior then for any $f \in H^m(\mathbb{R}^n)$ such that $f = 0$ on $K$, and 
any $\epsilon > 0$, there is a $\phi \in C^\infty_c(\mathbb{R}^n \setminus K)$ such that 
$\|f - \phi \|_{H^m(\mathbb{R}^n)} < \epsilon.$
\end{lemma}

\begin{proof}
See for example Theorem 5.29 of \cite{AF}.
\end{proof}


\medskip
\noindent
We need some more lemmas. 

\begin{lemma}\label{approx}
If $f$ is in $H^m(\mathbb{R}^n)$, $K$ is any compact subset of $\mathbb{R}^n$ and $\epsilon > 0$ 
we may choose a smooth cut-off function $\psi$ which is identically $1$ on $K$ and is such that if 
$\tilde f = \psi f$, then $\|f - \tilde f \|_{H^m(\mathbb{R}^n)} < \epsilon$.
\end{lemma}

\begin{proof}
If $\Phi$ is a standard normalised bump function which is identically one on the unit ball, and zero
off the ball of radius $2$ we have
$$\|f (1- \Phi(\cdot/R))\|_{H^m(\mathbb{R}^n)} \leq C  \sum_{j=0}^m {{m}\choose{j}} \int_{|x| \geq R } |D^j f|^2$$
and for $f$ in $H^m(\mathbb{R}^n)$ we have
$$ \sum_{j=0}^m {{m}\choose{j}} \int_{|x| \geq R } |D^j f|^2 \to 0 \mbox{ as } R \to \infty.$$
\end{proof}

\begin{lemma}\label{boundary}
Suppose that $f \in H^m(\mathbb{R}^n)$ and that $f =1$ on $K$ where $K$ has nonempty interior.
Suppose that $\psi$ is a smooth cut-off function which is identically $1$ on a neighbourhood of $K$. 
Let $\tilde f(x) = (f(x) -1) \psi(x)$.
Then $\tilde f \in  H^m_0(\mathbb{R}^n \setminus K)$. 
\end{lemma}

\begin{proof}
Clearly $\tilde f$ belongs to $H^m(\mathbb{R}^n)$ with $\|\tilde f\|_{H^m(\mathbb{R}^n)} 
\leq C (\| f\|_{H^m(\mathbb{R}^n)}+1)$ where $C$ depends on $\psi$ and $K$. Moreover $\tilde f = 0$ on $K$.
So by Lemma \ref{Sob}, $\tilde f$ belongs to $H^m_0(\mathbb{R}^n \setminus K)$ with the same control.
\end{proof}


\medskip
\noindent
For convenience we also recall the facts about traces we shall need from \cite{AF} or \cite{Evans}. 
(See also \cite{AH}.)
In the lemma which follows, $\mathcal{D}^j$ denotes any differential operator of order $j$ with constant 
coefficients. (We shall apply it in the special case that $\mathcal{D}^j = \Delta^{j/2}$ for $j$ even 
and $\mathcal{D}^j = (\nu \cdot \nabla)\Delta^{(j-1)/2}$ for $j$ odd, where $\nu$ is some unit vector.)

\begin{lemma}\label{trace}
Suppose $f \in H^m(\mathbb{R}^n)$ and that $1 \leq j \leq m-1$. Then
\begin{enumerate}
\item[(i)]
If $S$ is a piecewise smooth compact hypersurface, then $ \mathcal{D}^j f$ belongs to $H^{m-j-1/2}(S)$ and
$$ \|\mathcal{D}^j f|_{S} \|_{H^{m-j-1/2}(S)} \leq C \|f \|_{H^m(\mathbb{R}^n)};$$
moreover the map $f \mapsto \mathcal{D}^j f|_{S}$ varies continuously with small changes in $S$. 
\item[(ii)]
If $K$ is a compact convex set in $\mathbb{R}^n$ with nonempty interior and $f \in H^m_0(\mathbb{R}^n \setminus K)$, 
then  $\mathcal{D}^j f|_{\partial K} = 0$.

\end{enumerate}
\end{lemma}

\begin{proof}
(i) Since $f \in H^m(\mathbb{R}^n)$ we have $\mathcal{D}^j f \in H^{m-j}(\mathbb{R}^n)$ (with norm control) 
and so by the classical trace inequality its restriction to the hypersurface $S$ is in 
$H^{m-j-1/2}(S)$ provided $m-j \geq 1$, again with norm control. To see the continuity, suppose $S$ and $S'$ 
are nearby hypersurfaces, and take $\phi \in C^\infty_c(\mathbb{R}^n)$ and $\epsilon > 0$ such that 
$\| f - \phi \|_{H^m(\mathbb{R}^n)} < \epsilon$. Then
$$\|\mathcal{D}^j f|_{S} - \mathcal{D}^j f|_{S'}\|_{H^{m-j-1/2}}$$
$$\leq \| (\mathcal{D}^j f - \mathcal{D}^j \phi)|_{S}\|_{H^{m-j-1/2}}
 + \|\mathcal{D}^j \phi|_{S} - \mathcal{D}^j \phi|_{S'} \|_{H^{m-j-1/2}}
+ \| (\mathcal{D}^j \phi - \mathcal{D}^j) f|_{S'}\|_{H^{m-j-1/2}}.$$
The first and third terms are dominated by $\epsilon$ by the first part and the second term 
goes to zero as $S' \to S$ since $\phi$ is smooth. 

\medskip
\noindent
(ii) This is a consequence of the classical characterisation of the traces of functions in 
the Sobolev spaces $ H^m_0(\mathbb{R}^n \setminus K)$. Alternatively, it follows from the previous 
part since for smooth hypersurfaces contained in the interior of $K$ we will have $\mathcal{D}^j f|_S = 0$ for all $j \geq 1$.


\end{proof}

\section{The variational problem for the extremal energy}\label{sec_var}
\noindent
In this section we consider the variational problem
$$\inf \{ \|f\|_{H^m(\mathbb{R}^n)} \, : \,  f \in H^m(\mathbb{R}^n), \; f = g \mbox{ on } K\}$$
where $g$ is a prescribed member of $H^m(\mathbb{R}^n)$, $m \in \mathbb{N}$ and $K$ is a compact 
convex set in $\mathbb{R}^n$ with nonempty interior. We call this problem the $m$-extremal energy 
problem for $K$, or simply the extremal energy problem. 

\begin{proposition}\label{var}
The extremal energy problem 
$$\inf \{ \|f\|_{H^m(\mathbb{R}^n)} \, : \, f \in H^m(\mathbb{R}^n), \; f = g \; \rm{ on } \; K\}$$
has a unique solution.
\end{proposition}

\begin{proof}
This is just the elementary fact that in the Hilbert space $H^m(\mathbb{R}^n)$, there is a unique 
closest point to zero in the nonempty closed convex (in fact affine) set 
$\{ \|f\|_{H^m(\mathbb{R}^n)} \, : \, f \in H^m(\mathbb{R}^n), \; f = g \; {\rm{ on }} \; K\}$.
\end{proof}

\medskip
\noindent
When $K$ is a Euclidean ball 
we can see immediately that 
the unique solution to the extremal energy problem is radial since averaging any solution over rotations 
yields another solution which is radial.

\medskip
\noindent
By standard arguments (see for example \cite{Meckes14} for the case $m
= (n+1)/2$), the Euler-Lagrange equation for the extremal energy problem is 
$$ (I - \Delta)^m f = 0 \mbox{ on } \mathbb{R}^n \setminus K$$ 
in the weak sense (testing against functions
in $C^\infty_c(\mathbb{R}^n \setminus K)$) and so we are led to study the problem
\begin{eqnarray*}
\begin{aligned}
(I - \Delta)^m f &= 0 \mbox{ on } \mathbb{R}^n \setminus K \mbox{ in the weak sense}\\
f &= g \mbox{ on } K.
\end{aligned}
\end{eqnarray*}
We immediately have existence of solutions in $H^m(\mathbb{R}^n)$ to this problem 
from Proposition \ref{var}, and we also have uniqueness:

\begin{proposition}\label{uniqueness}
Suppose that $g \in H^{m}(\mathbb{R}^n)$. Then there is a unique solution to the problem
\begin{equation}\label{PDE}
\begin{aligned}
(I - \Delta)^m f &= 0 \mbox{ weakly on } \mathbb{R}^n \setminus K \\
f &= g \mbox{ on } K
\end{aligned}
\end{equation}
with $f \in  H^m(\mathbb{R}^n)$.
\end{proposition}

\begin{proof}
The space of solutions to \eqref{PDE} is $f_0 + \mathcal{M}$ where $f_0$ is any particular solution 
and $\mathcal{M}$ is the orthogonal complement of $H^m_0(\mathbb{R}^n \setminus K)$ in 
$\tilde{H}^m_0(\mathbb{R}^n \setminus K)$. But by Lemma \ref{Sob},  $\mathcal{M} = \{0\}$, 
and so the solution is unique.

\end{proof}


\medskip
\noindent
So the unique solution to the extremal energy problem of 
Proposition \ref{var} is also the unique solution to \eqref{PDE} in $H^m(\mathbb{R}^n)$.


\medskip
\noindent
We shall need in the next section to appeal to the theory of elliptic regularity.
Briefly, for an elliptic operator such as $(I - \Delta)^m$, weak solutions $f$ to $(I - \Delta)^m f = 0$ 
on an open set are actually smooth, and the equation $(I - \Delta)^m f = 0$ holds in the classical sense. See
\cite{AF} or \cite{Evans}.

\section{A formula for the $m$-extremal energy}\label{sec_formula}
\noindent
In this section we develop the promised formula for the $m$-extremal energy. See Theorem \ref{main} below.

\medskip
\noindent
We first consider arbitrary solutions $h \in H^m(\mathbb{R}^n)$ of the problem
\begin{equation*}
\left\{   
\begin{array}{llll}
(I - \Delta)^m &h &=& \; 0 \mbox{  weakly on } \mathbb{R}^n \setminus K \vspace{0.25cm}\\
&h &=& \; 1 \mbox{  on  } K
\end{array}     \right.
\end{equation*}
where $m \geq 1$ is an arbitrary integer and $K$ is a compact convex set with nonempty interior.
Consider a regularised distance function $d$ for $K$, which is defined
on $\mathbb{R}^n$ and satisfies  $d(x) \sim {\rm {dist}} (x, K)$. Let $K_r = 
\{ x \in \mathbb{R}^n \; : \; d(x) \leq r\}$ and let $\nu$ denote the unit normal 
pointing {\em out} of $K_r$. (Note that $\partial K_r$ is smooth and so Lemma \ref{trace} 
will be applicable.)

\medskip
\noindent
For $g \in H^m(\mathbb{R}^n)$ with compact support we have
\begin{eqnarray}\label{pp}
\begin{aligned}
 \langle g,h \rangle_{H^m(\mathbb{R}^n)} 
&= \sum_{j=0}^m {m \choose j} \int_{\mathbb{R}^n} D^j g \cdot D^jh \\
&= \int_K g + \sum_{j=0}^m {m \choose j} \lim_{r \downarrow 0} \int_{\{d(x) \geq r\}} D^j g \cdot D^jh
\end{aligned}
\end{eqnarray}
by the dominated convergence theorem.

\medskip
\noindent
We shall study the terms $\int_{\{d(x) \geq r\}} D^j g \cdot D^jh$ by integrating by parts.
In so doing, we shall systematically use Green's formulae
\begin{eqnarray*}
\begin{aligned}
&\int_\Omega \nabla \phi \cdot \nabla \psi 
&=& \; \; - \int_\Omega \phi \Delta \psi 
- \int_{\partial \Omega} \phi \frac{\partial \psi}{\partial \nu}\, {\rm d} S\\
\mbox{ and } 
\; \; &\int_\Omega (\Delta \phi) \psi 
&=& \; \; - \int_\Omega \nabla \phi \cdot \nabla \psi 
- \int_{\partial \Omega} \frac{\partial \phi }{\partial \nu} \psi \, {\rm d} S
\end{aligned}
\end{eqnarray*}
(where we are trying to increase the differentiability of $\psi$, which will be a function of $h$, 
with the eventual aim of using the equation satisfied by $h$,
and decrease that of $\phi$, which will be a function of $g$). In 
these formulae the region $\Omega$ will be $\{d(x) \geq r\}$ and $\nu$ is the unit 
normal pointing out of $K_r$. The smoothness of $h$ on $\mathbb{R}^n \setminus K$ coming from elliptic 
regularity, and the compact support of $g$ will ensure that each use of these formulae 
is valid. In particular the compact support of $g$ means that there are no boundary 
terms at infinity to consider.

\medskip
\noindent
To carry out these calculations it is convenient 
to define $\mathcal{D}^j f := \Delta^{j/2} f$ for $j$ even and 
$\mathcal{D}^j f := \frac{\partial}{\partial \nu}\Delta^{(j-1)/2} f$ for $j$ odd.
(Note the distinction here when $j$ is odd between $\mathcal{D}$ and
$D$. While a risking a possible typographical confusion, at a first
reading one should simply think of $\mathcal{D}^j$ and $D^j$ as denoting
suitable derivatives of order $j$. Note also that 
these boundary operators occur in the PDE literature in problems involving hybrid Dirichlet--Navier
boundary conditions. See for example \cite{Navier}, p.33. It is in the
lemma below that expressions such as 
$\Delta u$ (rather than  $\frac{\partial^2 u}{\partial \nu^2}$) appear more naturally on the boundary.

\begin{lemma}\label{ibp} Let $K \subseteq \mathbb{R}^n$ be compact and convex.
For $1 \leq j  \leq m$ and $g \in H^m(\mathbb{R}^n)$ with compact support we have 
$$ \int_{\{d(x) \geq r\}} D^j g \cdot D^jh = (-1)^j\left(\int_{\{d(x) \geq r\}} g \Delta^j h + 
\sum_{k=0}^{j-1} (-1)^k \int_{\partial K_r} \mathcal{D}^k g  \mathcal{D}^{2j-k-1} h \, {\rm d} S\right).$$
\end{lemma}

\begin{proof}
Let us first consider the case $j=1$. Then
\begin{eqnarray*}
\begin{aligned}
\int_{\{d(x) \geq r\}} \nabla g \cdot \nabla h &= - \int_{\{d(x) \geq r\}} g \Delta h 
- \int_{\partial K_r} g \frac{\partial h}{\partial \nu}\, {\rm d} S \\ 
&= -\left(\int_{\{d(x) \geq r\}} g \Delta h 
+ \int_{\partial K_r} g \mathcal{D} h\, {\rm d} S \right)
\end{aligned}
\end{eqnarray*}
as required.

\medskip
\noindent
Now, for $j=2$ we have
\begin{eqnarray*}
\begin{aligned}
\int_{\{d(x) \geq r\}} \Delta g \Delta h &= - \int_{\{d(x) \geq r\}} \nabla g \cdot \nabla 
\Delta h - \int_{\partial K_r}  \frac{\partial g}{\partial \nu} \Delta h  {\rm d} S \\
&=  \int_{\{d(x) \geq r\}} g \Delta^2 h  + \int_{\partial K_r} g  \frac{\partial \Delta h}
{\partial \nu} \, {\rm d} S -  \int_{\partial K_r}  \frac{\partial g}{\partial \nu} \Delta h  {\rm d} S\\
&=  \int_{\{d(x) \geq r\}} g \Delta^2 h + \int_{\partial K_r} g  \mathcal{D}^3 h \, {\rm d} S 
-  \int_{\partial K_r}  \mathcal{D} g  \mathcal{D}^2 h {\rm d} S 
\end{aligned}
\end{eqnarray*}
as required, noting that the term $\frac{\partial g}{\partial \nu}$ makes sense as a member of $H^{m-3/2}$ 
of the hypersurface $\partial K_r$ by the trace inequality, Lemma \ref{trace}. 
(For this case to occur we must have $m \geq 2$.)

\medskip
\noindent
Similarly, for $j=3$, we have
\begin{eqnarray*}
\begin{aligned}
&\int_{\{d(x) \geq r\}} \nabla \Delta g \cdot \nabla \Delta h \\ 
&=  - \int_{\{d(x) \geq r\}} \Delta g \Delta^2 h - \int_{\partial K_r} \Delta g
\frac{\partial \Delta h}
{\partial \nu} \, {\rm d} S \\
&= \int_{\{d(x) \geq r\}}  \nabla g \cdot \nabla \Delta^2 h 
+ \int_{\partial K_r} \frac{\partial g}{\partial \nu}\Delta^2 h  \, {\rm d} S
- \int_{\partial K_r} \Delta g \frac{\partial \Delta h}{\partial \nu} \, {\rm d} S \\
&= -  \int_{\{d(x) \geq r\}} g \Delta^3 h 
- \int_{\partial K_r} g \frac{\partial \Delta^2 h} {\partial \nu} \, {\rm d} S 
+ \int_{\partial K_r} \frac{\partial g}{\partial \nu}\Delta^2 h  \, {\rm d} S
- \int_{\partial K_r} \Delta g \frac{\partial \Delta h} {\partial \nu} \, {\rm d} S\\
&= - \left( \int_{\{d(x) \geq r\}} g \Delta^3 h 
+ \int_{\partial K_r} g \mathcal{D}^5 h \, {\rm d} S 
- \int_{\partial K_r} \mathcal{D} g\mathcal{D}^4 h  \, {\rm d} S
+ \int_{\partial K_r} \mathcal{D}^2g \mathcal{D}^3 h  \, {\rm d} S \right)
\end{aligned}
\end{eqnarray*}
noting that the terms $\frac{\partial g}{\partial \nu}$ and $\Delta g$ make sense as members
of $H^{m-3/2}$ and  $H^{m-5/2}$ respectively of the hypersurface $\partial K_r$ by the trace inequality, 
Lemma \ref{trace}. (For this case to occur we must have $m \geq 3$.)

\medskip
\noindent
Continuing, we see that the general case takes the form 
$$\int_{\{d(x) \geq r\}} D^j g \cdot D^j h$$
\begin{eqnarray*}
\begin{aligned}
= (-1)^j \Big\{\int_{\{d(x) \geq r\}} g \Delta^j h 
&+ \int_{\partial K_r} g \frac{\partial \Delta^{j-1} h} {\partial \nu} \, {\rm d} S
- \int_{\partial K_r} \frac{\partial g}{\partial \nu}\Delta^{j-1} h  \, {\rm d} S \\
&+ \int_{\partial K_r} \Delta g \frac{\partial \Delta^{j-2} h} {\partial \nu} \, {\rm d} S 
- \int_{\partial K_r} \frac{\partial \Delta g}{\partial \nu}\Delta^{j-2} h  \, {\rm d} S
+ \dots \Big\}
\end{aligned}
\end{eqnarray*}
where the last term inside the curly brackets is
$$  - \int_{\partial K_r} \frac{\partial \Delta^{(j-2)/2} g}{\partial \nu}\Delta^{j/2} h  \, {\rm d} S $$
when $j$ is even and
$$ + \int_{\partial K_r} \Delta^{(j-1)/2} g \frac{\partial \Delta^{(j-1)/2} h} {\partial \nu} \, {\rm d} S $$
when $j$ is odd. That is,
$$ \int_{\{d(x) \geq r\}} D^j h \cdot D^jg$$
\begin{eqnarray*}
\begin{aligned}
 =(-1)^j\Big\{\int_{\{d(x) \geq r\}} g \Delta^j h 
&+ \int_{\partial K_r} g \mathcal{D}^{(2j-1)} h\, {\rm d} S
- \int_{\partial K_r} \mathcal{D}g \mathcal{D}^{(2j-2)} h  \, {\rm d} S\\
&+ \int_{\partial K_r} \mathcal{D}^2 g \mathcal{D} ^{(2j-3)} h \, {\rm d} S 
- \int_{\partial K_r} \mathcal{D}^3 g \mathcal{D} ^{(2j-4)} h \, {\rm d} S \\
&+ \dots + (-1)^{j-1} \int_{\partial K_r} \mathcal{D}^{(j-1)} g \mathcal{D} ^{j} h \, {\rm d} S \Big\}.
\end{aligned}
\end{eqnarray*}
Note that the terms $\mathcal{D}g, \mathcal{D}^2 g, \dots , \mathcal{D}^{(j-1)} g$ 
make sense as members of $H^{m-3/2}$, $H^{m-5/2}$, $\dots, H^{m-j+1/2}$ respectively of the hypersurface 
$\partial K_r$ by the trace inequality, Lemma \ref{trace}. (For this case to occur we must have $m \geq j$.) This 
establishes the lemma.
\end{proof}

\noindent
If we take the identities of Lemma \ref{ibp}, multiply by ${m \choose j}$ and sum from $j=1$ to $m$ 
we get
\begin{eqnarray*}
\begin{aligned}
&\sum_{j=1}^m {m \choose j}  \int_{\{d(x) \geq r\}} D^j g \cdot D^jh \\
= &\sum_{j=1}^m {m \choose j}  (-1)^j\left(\int_{\{d(x) \geq r\}} g \Delta^j h + 
\sum_{k=0}^{j-1} (-1)^k \int_{\partial K_r} \mathcal{D}^k g  \mathcal{D}^{2j-k-1} h \, {\rm d} S\right)\\
= &\int_{\{d(x) \geq r\}} g \left(\sum_{j=1}^m {m \choose j}  (-1)^j  \Delta^j h\right) 
+ \sum_{j=1}^m {m \choose j}  (-1)^j\sum_{k=0}^{j-1} (-1)^k \int_{\partial K_r} \mathcal{D}^k g  
\mathcal{D}^{2j-k-1} h \, {\rm d} S.
\end{aligned}
\end{eqnarray*}

\medskip
\noindent
Now add $ \int_{\{d(x) \geq r\}} gh$ (corresponding to the the term $j=0$) to both sides to obtain
\begin{eqnarray*}
\begin{aligned}
&\sum_{j=0}^m {m \choose j}  \int_{\{d(x) \geq r\}} D^j h \cdot D^jg\\
= &\int_{\{d(x) \geq r\}} g \left(\sum_{j=0}^m {m \choose j}  (-1)^j  \Delta^j h\right) 
+ \sum_{0 \leq k < j \leq m}(-1)^{j+k}{m \choose j} \int_{\partial K_r} \mathcal{D}^k g  
\mathcal{D}^{2j-k-1} h \, {\rm d} S\\
= &\int_{\{d(x) \geq r\}} g (I - \Delta)^m h + \sum_{0 \leq k < j \leq m}(-1)^{j+k}{m \choose j}\int_{\partial K_r} \mathcal{D}^k g  \mathcal{D}^{2j-k-1} h \, {\rm d} S\\
= &\sum_{0 \leq k < j \leq m}(-1)^{j+k}{m \choose j}\int_{\partial K_r} \mathcal{D}^k g  
\mathcal{D}^{2j-k-1} h \, {\rm d} S
\end{aligned}
\end{eqnarray*}
since $h$ satisfies the equation  $(I - \Delta)^m h =0$ in the classical sense on ${\{d(x) >0 \}}$ by 
elliptic regularity.

\medskip
\noindent
Combining this with \eqref{pp}, we obtain the representation
$$ \langle g,h \rangle_{H^m(\mathbb{R}^n)} 
= \int_K g
+ \lim_{r \downarrow 0}\sum_{0 \leq k < j \leq m}(-1)^{j+k}{m \choose j}\int_{\partial K_r} \mathcal{D}^k g  
\mathcal{D}^{2j-k-1} h \, {\rm d} S$$
valid for any $g \in H^m(\mathbb{R}^n)$ with compact support. 

\medskip
\noindent
We would like to extend this formula to be valid for {\em all}
$g \in  H^m(\mathbb{R}^n)$. Observe that all the terms on the right hand side make sense 
for such $g$ -- as was noted during the proof of Lemma \ref{ibp}. In particular there is no problem
with the existence of the limit as $r \downarrow 0$ for such $g$ -- simply multiply it by a smooth 
cut-off which is identically $1$ on $\{ d(x) \leq 2\}$ to obtain $\tilde{g}$ of compact support for which 
the corresponding limit exists and the value of which is left unaltered by the multiplication.  

\medskip
\noindent
So the linear functional $\Lambda$ given by 
$$ \Lambda g :=  \langle g,h \rangle_{H^m(\mathbb{R}^n)} - \int_K g
- \lim_{r \downarrow 0}\left(\sum_{0 \leq k < j \leq m}(-1)^{j+k} {m \choose j}\int_{\{d(x) = r\}} \mathcal{D}^k g  
\mathcal{D}^{2j-k-1} h \, {\rm d} S\right)$$
is well-defined on $H^m(\mathbb{R}^n)$ and is identically zero on the dense subspace consisting of functions 
of compact support. If $g$ is in $H^m(\mathbb{R}^n)$ and $\epsilon > 0$ we may choose by Lemma \ref{approx} 
a cut-off function $\psi$ which is identically $1$ on $\{d(x) \leq 2\}$ and such that if 
$\tilde g = \psi g$, then
$\|g - \tilde g \|_{H^m(\mathbb{R}^n)} < \epsilon$.
Hence 
$$ \Lambda g = \Lambda g - \Lambda \tilde g  = \Lambda( g - \tilde g) 
=  \langle g-\tilde g, h\rangle_{H^m(\mathbb{R}^n)} $$
so that 
$$|\Lambda g| \leq \| g - \tilde g\|_{H^m(\mathbb{R}^n)}\|h\|_{H^m(\mathbb{R}^n)} < \epsilon  \|h\|_{H^m(\mathbb{R}^n)}.$$
Thus $\Lambda g =0$ for all $g \in H^m(\mathbb{R}^n)$ and we have:

\begin{proposition}\label{rep}
For $m \in \mathbb{N}$, $K$ any compact convex set in $\mathbb{R}^n$ and any solution $h \in H^m(\mathbb{R}^n)$
to the problem $(I-\Delta)^m h = 0$ weakly on $\mathbb{R}^n \setminus K$, we have the 
formula
\begin{equation}\label{ytr}
 \langle g,h \rangle_{H^m(\mathbb{R}^n)} 
= \int_K g
+ \lim_{r \downarrow 0}\sum_{0 \leq k < j \leq m}(-1)^{j+k}{m \choose j}\int_{\partial K_r} \mathcal{D}^k g  
\mathcal{D}^{2j-k-1} h \, {\rm d} S,
\end{equation}
valid for every $g \in H^m(\mathbb{R}^n)$.
\end{proposition}

\medskip
\noindent
It is intuitively reasonable that for functions $g \in H^m_0(\mathbb{R}^n \setminus K)$ there should be no 
boundary terms in formula \eqref{ytr} and indeed for such $g$ it is natural that $\int_K g = 0$
and also that $\langle g, h \rangle_{H^m} = 0$ using the equation. This we establish next. (Note that in the 
following lemma we had better use the equation satisfied by $h$ as otherwise there is 
no reason to believe the terms involving higher derivatives of $h$ exist.)

\begin{lemma}\label{limbound}
For $g \in H^m_0(\mathbb{R}^n \setminus K)$ and $h \in H^m(\mathbb{R}^n)$ 
any solution to  $(I-\Delta)^m h = 0$ weakly on $\mathbb{R}^n \setminus K$ we have
$$ \lim_{r \downarrow 0}\sum_{0 \leq k < j \leq m}(-1)^{j+k}{m \choose j} \int_{\partial K_r} \mathcal{D}^k g  
\mathcal{D}^{2j-k-1} h \, {\rm d} S = 0.$$
On the other hand, if $g \in H^m(\mathbb{R}^n)$, if $g = 0$ on $K$ (i.e. $g = 0$ pointwise on $K$ 
if $m > n/2$, and is zero on int $K \neq \emptyset$ when $m \leq n/2$) and if
$$ \lim_{r \downarrow 0}\sum_{0 \leq k < j \leq m}(-1)^{j+k}{m \choose j} \int_{\partial K_r} \mathcal{D}^k g  
\mathcal{D}^{2j-k-1} h \, {\rm d} S = 0$$
for every solution $h \in H^m(\mathbb{R}^n)$ to  $(I-\Delta)^m h = 0$ weakly on $\mathbb{R}^n \setminus K$
which also satisfies $h=0$ on $K$,
then $g \in H^m_0(\mathbb{R}^n \setminus K)$.
\end{lemma}

\begin{proof}
Consider the first assertion. By our formula \eqref{ytr} we have 
$$ \lim_{r \downarrow 0}\sum_{0 \leq k < j \leq m}(-1)^{j+k}{m \choose j}\int_{\partial K_r} \mathcal{D}^k g  
\mathcal{D}^{2j-k-1} h \, {\rm d} S =  \langle g,h \rangle_{H^m(\mathbb{R}^n)} - \int_K g .$$
So it is enough to show that 
$$ \Lambda g := \langle g,h \rangle_{H^m(\mathbb{R}^n)} - \int_K g$$
satisfies 
$$ \Lambda g = 0$$
for all $g \in H^m_0(\mathbb{R}^n \setminus K)$.
Let $\epsilon > 0$ and $\tilde g \in C^\infty_c(\mathbb{R}^n \setminus K)$ be such that 
$\| g - \tilde g \|_{H^m_0(\mathbb{R}^n \setminus K)} < \epsilon$.
Clearly we have $\Lambda \tilde g = 0$ since $h$ satisfies $(I-\Delta)^m h =0$ weakly on $\mathbb{R}^n \setminus K$ and 
$K \cap \mathbb{R}^n \setminus K = \emptyset$, and so
$$ \Lambda g  = \Lambda g - \Lambda \tilde g = \Lambda( g - \tilde g) 
=  \langle g-\tilde g,h \rangle_{H^m(\mathbb{R}^n)} - \int_K (g - \tilde g)h,$$
thus 
$$|\Lambda g | \leq \| g - \tilde g\|_{H^m(\mathbb{R}^n)} \|h\|_{H^m(\mathbb{R}^n)} + \|g - \tilde g\|_2 \, \|h\|_2 < 
\epsilon \|h\|_{H^m(\mathbb{R}^n)}.$$
Hence $\Lambda g =0$ for all $g \in H^m_0(\mathbb{R}^n \setminus K)$.

\medskip
\noindent
For the second assertion, if the limit in question is zero, we have by the representation formula \eqref{ytr} that 
$\langle g,h \rangle_{H^m} = 0$ for all $h \in H^m(\mathbb{R}^n)$ which are solutions to  $(I-\Delta)^m h = 0$ weakly 
on $\mathbb{R}^n \setminus K$ which also satisfy $h=0$ on $K$. So 
$g \in \tilde{H}^m_0(\mathbb{R}^n \setminus K)$ and $g$ is perpendicular to the space $\mathcal{M}$ 
introduced above in Proposition \ref{uniqueness}. By definition of $\mathcal{M}$ this means that 
$g \in H^m_0(\mathbb{R}^n \setminus K)$.

\end{proof}


\medskip
\noindent
We have not yet used the requirement that $K$ have nonempty interior nor used any boundary conditions 
satisfied by $h$ (except to obtain the formula \eqref{pp}). In order to deduce a useful formula for 
$\|h\|^2_{H^m(\mathbb{R}^n)}$, we shall have to do so,
and we shall use Lemma \ref{boundary} freely from now on. 

\medskip
\noindent
So, let $H$ be any function in $H^m(\mathbb{R}^n)$ such that $H=1$ on $K$. Then, by 
Lemma \ref{boundary} we have that  $g(x) = (H(x) - 1) \psi(x)$ belongs to $H^m_0(\mathbb{R}^n \setminus K)$
(where $\psi$ is any smooth function of compact support which is identically $1$ on $\{d(x) \leq 2\}$), 
and so we can apply Lemma \ref{limbound} to obtain
\begin{equation*}
\begin{aligned}
&\lim_{r \downarrow 0}\sum_{0 \leq k < j \leq m}(-1)^{j+k}{m \choose j}\int_{\partial K_r} \mathcal{D}^k H  
\mathcal{D}^{2j-k-1} h \, {\rm d} S \\
- &\lim_{r \downarrow 0}\sum_{0 \leq k < j \leq m}(-1)^{j+k}{m \choose j} \int_{\partial K_r} \mathcal{D}^k \psi  
\mathcal{D}^{2j-k-1} h \, {\rm d} S = 0.
\end{aligned}
\end{equation*}
But for $r$ sufficiently close to zero we have 
\begin{equation*}
\begin{aligned}
&\sum_{0 \leq k < j \leq m} (-1)^{j+k}{m \choose j} \int_{\partial K_r} \mathcal{D}^k \psi  
\mathcal{D}^{2j-k-1} h \, {\rm d} S \\
& = \sum_{j=1}^{m} {m \choose j}  (-1)^j \int_{\partial K_r} \mathcal{D}^{2j-1} h \, {\rm d} S
\end{aligned}
\end{equation*}
so that 
\begin{eqnarray}\label{star}
\begin{aligned}
&\lim_{r \downarrow 0}\sum_{0 \leq k < j \leq m}(-1)^{j+k}{m \choose j} \int_{\partial K_r} \mathcal{D}^k H  
\mathcal{D}^{2j-k-1} h \, {\rm d} S \\
= &\lim_{r \downarrow 0}\sum_{j=1}^m  (-1)^j{m \choose j} \int_{\partial K_r} \mathcal{D}^{2j-1} h \, {\rm d} S.
\end{aligned}
\end{eqnarray}

\medskip
\noindent
Next, we take $g = H$ in Proposition \ref{rep} and apply \eqref{star} to obtain:
\begin{proposition}\label{kjh}
Suppose $K$ has nonempty interior, $h$ is any solution in $H^m(\mathbb{R}^n \setminus K)$ to 
$(I-\Delta)^m h = 0$ weakly on $\mathbb{R}^n \setminus K$, and $H$ is any function in $H^m(\mathbb{R}^n)$ 
such that $H=1$ on $K$. Then we have
$$ \langle H, h \rangle_{H^m(\mathbb{R}^n)} 
= {\rm Vol}(K) + \lim_{r \downarrow 0}\sum_{j=1}^m  (-1)^j {m \choose j} 
\int_{\partial K_r} \mathcal{D}^{2j-1} h \, {\rm d} S.$$  
\end{proposition}

\noindent
Finally, we claim that the terms with $j \leq m/2$ give a contribution 
of zero in this formula. 

\medskip
\noindent
Indeed, if we have $h \in H^m(\mathbb{R}^n)$ and $h = 1$ on $K$,  Lemma \ref{boundary} tells us that 
for suitable smooth $\psi$ of compact support, identically $1$ on a neighbourhood of $K$ we have 
$g = (h-1)\psi \in H^m_0(\mathbb{R}^n \setminus K)$.
So by Lemma \ref{trace}, $\mathcal{D}^i g $ is the zero member of
$H^{m-i-1/2}(\partial K)$ for $0 \leq i \leq m-1$, and continuity of traces shows that we therefore have
$$ \int_{\partial K_r } \mathcal{D}^i g {\rm d}S \to 0 \mbox{ as } r \downarrow 0$$
for $0\leq i \leq m-1$. 


\medskip
\noindent
In particular, since  $\mathcal{D}^i h = \mathcal{D}^i g + \mathcal{D}^i \psi$
near $\partial K$ we have
$$ \int_{\partial K_r } \mathcal{D}^i h {\rm d}S \to 0 \mbox{ as } r \downarrow 0$$
for $1 \leq i \leq m-1$. So 
$$ \int_{\partial K_r } \mathcal{D}^{2j-1} h {\rm d}S \to 0 \mbox{ as } r \downarrow 0$$
for $1 \leq 2j-1 \leq m-1$, that is for $1 \leq j \leq m/2$.

\medskip
\noindent
Hence
$$ \langle H, h \rangle_{H^m(\mathbb{R}^n)} = {\rm {Vol}} (K) + \lim_{r \downarrow 0}\sum_{m/2 < j \leq m} (-1)^j  {m \choose j} 
\int_{\partial K_r} \mathcal{D}^{2j-1} h \, {\rm d} S,$$  
and we have proved:
 
\begin{theorem}\label{main}
Suppose that $K$ is a compact convex set in $\mathbb{R}^n$ which 
has nonempty interior. Let $H$ be any function in $H^m(\mathbb{R}^n)$
such that $H = 1$ on $K$. Then the unique solution $h \in H^m(\mathbb{R}^n)$ to the problem 
$$ (I - \Delta)^m h = 0 \mbox{ weakly on } \mathbb{R}^n \setminus K$$
$$ h = 1 \mbox{ on } K$$
satisfies
$$ \langle H, h\rangle_{H^m(\mathbb{R}^n)} = {\rm Vol}(K) + \sum_{m/2 < j \leq m} 
 (-1)^{j} {{m} \choose j}
\lim_{r \downarrow 0} \int_{\partial K_r} \frac{\partial}{\partial \nu} \Delta^{j-1} h {\rm d}S. $$
In particular, taking $H = h$,
\begin{equation}\label{mainformula}
\|h \|_{H^{m}(\mathbb{R}^n)}^2 = {\rm Vol}(K) + \sum_{m/2 < j \leq m} 
 (-1)^{j} {{m} \choose j}
\lim_{r \downarrow 0} \int_{\partial K_r} \frac{\partial}{\partial \nu} \Delta^{j-1} h {\rm d}S.
\end{equation}
\end{theorem}

\medskip
\noindent
If the boundary of $K$ is sufficiently regular to allow the invocation 
of boundary regularity for elliptic equations we may realise this formula more succinctly as
$$\|h \|_{H^{m}(\mathbb{R}^n)}^2 = {\rm Vol}(K) + \sum_{m/2 < j \leq m} 
 (-1)^{j} {{m} \choose j} \int_{\partial K} \frac{\partial}{\partial \nu} \Delta^{j-1} h {\rm d}S. $$

\medskip
\noindent
The main virtue of Theorem \ref{main} is that it explicitly demonstrates the dependence of 
the $m$-extremal energy upon the volume of $K$ and integrals over its boundary. It is analogous to the representation
of the classical capacity $\inf\{ \int |\nabla u|^2 \, : u = 1 \mbox{ on } K \}$ as $\int_{\partial K} \frac{\partial u}{\partial \nu} {\rm d}S$ for $u$ satisfying $\Delta u = 0$ off $K$, $u = 1$ on $\partial K$ and $u \to 0$ at $\infty$. The fact that 
$\|h \|_{H^{m}(\mathbb{R}^n)}^2 =  \langle h, H\rangle_{H^m(\mathbb{R}^n)}$ for any $H \in H^m(\mathbb{R}^n)$ 
which is identically $1$ on $K$ when $m=(n+1)/2$ is more general and due to Meckes \cite{Meckes14}, 
Proposition 4.2. One may wonder whether the alternating sum in Theorem \ref{main} reflects  
an underlying ``Euler characteristic'', and whether the individual terms appearing might have some separate 
meaning.
It is also of interest to determine whether or not the solutions of the boundary value 
problems are always positive. See for example \cite{Navier} for a discussion of such matters.

\medskip
\noindent
As mentioned in the Introduction, Leinster and Meckes have an alternative approach to calculating 
magnitudes. Indeed, Theorem 4.16 of \cite{LM16}, 
specialised to the case of Euclidean spaces, shows that up to a constant, the
magnitude of a compact subset of $\mathbb{R}^n$ is the integral of its potential function, provided the 
potential function is in $L^1(\mathbb{R}^n)$. In the context of Theorem \ref{main} and formula \eqref{mainformula}
this translates simply as 
$$ \|h\|_{H^{m}(\mathbb{R}^n)}^2 = \int_{\mathbb{R}^n} h.$$
While this formula arguably offers some comptational advantages over formula \eqref{mainformula}, it does not bring 
out the important role that the boundary of $K$ plays. In this regard see also the last remark in 
Section \ref{sec_final}.


\section{Combinatorial preliminaries}\label{combsec}
\medskip
\noindent
For our discussion of solutions to certain ordinary differential equations we shall need 
to consider some special polynomials which have a combinatorial flavour.

\medskip
\noindent
We define $g_0(t) := 1$ and, given $g_j$, we define $g_{j+1}$ by
$$ g_{j+1}(t) := t^3 g_j'(t) + t g_j(t).$$

\begin{lemma}
For $j \geq 1$ we have
$$ g_j(t) = \sum_k c^j_k t^k$$ 
where $c^j_k = 0$ for $k < j$ and $k \geq 2j$, $c^j_j = 1$ for all $j$, and each $c^j_k$ for $j \leq k \leq 
2j-1$ is a nonnegative integer.
\end{lemma}

\noindent
The easy proof is left to the reader. The defining formula for $g_j$ leads us immediately to a recurrence relation 
for the coefficients $c^j_k$: 

\begin{lemma}\label{comb}
For $j+1 \leq k \leq 2j-1$ we have 
\begin{equation}\label{pascal}
c^{j+1}_{k+1} = (k-1)c^j_{k-1} + c^j_k
\end{equation}
and
\begin{equation}\label{pascal1}
c^{j+1}_{2j+1}= (2j-1) c^j_{2j-1}.
\end{equation}
Moreover for  $j+1 \leq k \leq 2j-1$ we have
$$2j c^{j+1}_{k+1} = (k-1)kc^j_{k-1} + 2kc^j_k.$$
\end{lemma}

\begin{proof}
We have by definition
\begin{equation*}
\begin{aligned}
g_{j+1}(t) &= t^3\left(\sum_{k=j}^{2j-1} c^j_k k t^{k-1}\right) + t \left(\sum_{k=j}^{2j-1} c^j_k  t^{k}\right)\\
&=  \sum_{k=j}^{2j-1} c^j_k k t^{k+2} + \sum_{k=j}^{2j-1} c^j_k  t^{k+1} 
=  \sum_{k=j+1}^{2j} c^j_{k-1} (k-1) t^{k+1} + \sum_{k=j}^{2j-1} c^j_k  t^{k+1}\\
&= c^j_j t^{j+1} +  \sum_{k=j+1}^{2j-1}  (c^j_{k-1} (k-1) + c^j_k)t^{k+1} + c^{j}_{2j-1}(2j-1)t^{2j+1}.
\end{aligned}
\end{equation*}

On the other hand, 
$$ g_{j+1}(t) = \sum_{k=j+1}^{2j+1}c^{j+1}_kt^k = \sum_{k=j}^{2j}c^{j+1}_{k+1}t^{k+1},$$
and then \eqref{pascal} and \eqref{pascal1} follow by comparing the two expressions.
The final identity follows readily from the expression \eqref{expl} below.

\end{proof}
\medskip
\noindent
Pictorially, identity \eqref{pascal} can be represented as a Pascal triangle as follows, (with the first row
representing $c_1^1$, the second $c^2_2, c^2_3$, etc.):

\begin{eqnarray*}
\begin{aligned}
&1 \\
1 \;& \; \; \; 1 \\
1 \; \; \; &3 \; \; \; \, 3 \\
1 \; \; \; 6 \; & \; \; 15 \;  \; \; 15 \\
1 \; \; \, 10 \; \; \, 4&5 \; \;  105 \; \; 105\\ 
1 \; \; \;  15 \;\;  105 & \; 420 \; 945 \; 945 \\
\vdots
\end{aligned}
\end{eqnarray*}
An explicit formula for $c^j_k$ for $j < k \leq 2j-1$ is given by
\begin{equation}\label{expl}
 c^j_k = \frac{(k-1)(k-2) \cdots (2j-k)}{2^{k-j} (k-j)!}
\end{equation}
and this follows immediately from equations \eqref{pascal} and \eqref{pascal1} together with the fact that 
$c^j_j = 1$ for all $j$. 

\medskip
\noindent
The numbers $c^j_k$ are reminiscent of Stirling numbers and are sometimes called Bessel numbers of the first kind; 
the $g_j$ are related to Grosswald's polynomials. See \cite{grosswald} and \cite{web}. The $g_j$ 
can also be expressed in terms of generalised Laguerre polynomials and hypergeometric functions; we thank Jim Wright, 
Chris Smyth and Adri Olde Daalhuis for pointing these connections out to us.



\section{Explicit radial solutions to $(I - \Delta)^m h = 0$}\label{Sec_ODE}
\noindent
When $K$ is a Euclidean ball, the unique solution to problem \eqref{PDE} with $g = 1$ is necessarily radial
(since an average over rotates of a solution is also a solution). 
So in this section we seek radial solutions $h \in H^m(\mathbb{R}^n)$ to
$$ (I- \Delta)^m h = 0 \mbox{ on } \{|x| > R\}$$
where $R > 0$. Then we shall impose the extra (boundary) condition $h = 1$ on $\{|x| \leq R\}$. 
We are particularly interested in the case where $n$ is odd and
$m = (n+1)/2$. We shall systematically abuse notation and not distinguish between the function 
$h(|\cdot|)$ defined on $\{|x| > R\} $ in $\mathbb{R}^n$ and the function $h(\cdot)$ defined on $(R, \infty)$.

\medskip
\noindent
Recall that the action of the Laplacian in $\mathbb{R}^n$ on radial functions is given by
$$ \Delta f(r) = f''(r) + \frac{n-1}{r}f'(r)$$
and so we define, for an integer $\nu \geq 0$,  
$$\Delta_\nu f := f'' + \frac{2 \nu}{r} f'$$
and consider solutions of 
\begin{equation}\label{fds}
(I - \Delta_\nu)^m h = 0
\end{equation} 
on $(R, \infty)$ where $R > 0$, with $h(|\cdot|) \in L^2(\mathbb{R}^n)$. 
(We hope that the the use of $\nu$ both as a proxy for $(n-1)/2$ in this section and as a normal 
direction to $K$ in previous sections will not cause confusion.)
If $n$ is odd and $\nu=(n-1)/2$ this amounts 
to finding radial solutions in $L^2$ of 
$$ (I - \Delta)^m h = 0  \mbox{ on } |x| > R.$$

\medskip
\noindent
Let us set up some notation. We define, for $k \geq 0$,
$$ f_k(r) := \frac{e^{-r}}{r^k}.$$
We define $\psi_0(r) = e^{-r}$ and, for $j \geq 1$
$$ \psi_j(r) := \sum_{k=j}^{2j-1} c^j_k f_k(r) = e^{-r} g_j(1/r)$$
where the $g_j$ are defined in Section \ref{combsec}.

\medskip
\noindent
From Pascal's triangle for the coefficients $c^j_k$, (Lemma \ref{comb}), we can write down the 
first few $\psi_j$ explicitly: 
\begin{eqnarray*}
\begin{aligned}
\psi_0 = f_0 &= e^{-r}, \\
\psi_1 = f_1 &= e^{-r}/r, \\
\psi_2 = f_2 + f_3 &= e^{-r}(r^{-2} + r^{-3}),\\
\psi_3 = f_3 + 3 f_4 +3 f_5 &= e^{-r}(r^{-3} + 3 r^{-4} + 3 r^{-5}),\\ 
\psi_4 = f_4 + 6 f_5 + 15 f_6 + 15 f_7 &= e^{-r}(r^{-4} + 6 r^{-5} + 15 r^{-6} + 15r^{-7}),\\ 
\psi_5 = f_5 + 10 f_6 + 45 f_7 + 105 f_8 + 105 f_9  &= e^{-r}(r^{-5}
+ 10 r^{-6} + 45 r^{-7} + 105r^{-8} +105r^{-9} )\\
\vdots
\end{aligned}
\end{eqnarray*}

\begin{lemma}\label{diff}
For $k \geq 0$ and $\nu \geq 0$ we have
$$ (\Delta_\nu - I)f_k = 2(k - \nu)f_{k+1} + k(k+1-2\nu)f_{k+2}.$$
In particular, 
$$ (\Delta_\nu - I)f_0 = -2 \nu f_{1}.$$
\end{lemma}

\begin{proof}
Let us calculate 
$$f_k' = -\frac{e^{-r}}{r^k} - k \frac{e^{-r}}{r^{k+1}} = -f_k - k f_{k+1}$$
so that 
\begin{equation*}
\begin{aligned}
f_k'' &= -f_k' - k f'_{k+1}\\
& = -(-f_k - kf_{k+1}) -k(-f_{k+1} - (k+1)f_{k+2}) \\
&=  f_k + 2k f_{k+1} +k(k+1)f_{k+2}.
\end{aligned}
\end{equation*}
Therefore
\begin{equation*}
\begin{aligned}
(\Delta_\nu - I)f_k &=  f_k  + 2k f_{k+1} +k(k+1)f_{k+2} -
  2\nu\left(\frac{f_k}{r} + \frac{kf_{k+1}}{r}\right) -f_k \\
 &= 2k f_{k+1} +k(k+1)f_{k+2} - 2\nu\left(f_{k+1} + k f_{k+2}\right)\\
 &=  2(k- \nu)f_{k+1} + k(k+1-2\nu)f_{k+2}.
\end{aligned}
\end{equation*}

\end{proof}

\begin{proposition}\label{key}
For $j \geq 0$ and $\nu \geq 0$ we have
$$(\Delta_\nu -I) \psi_j  = 2(j -\nu)\psi_{j+1}.$$
In particular, 
$$(\Delta_\nu -I) \psi_\nu  = 0.$$

\end{proposition}

\begin{proof}
The case $j=0$ is a special case of Lemma \ref{diff}. For $j \geq 1$
we have  
\begin{equation*}
\begin{aligned}
(\Delta_\nu -I) \psi_j &= (\Delta_\nu -I) \sum_{k=j}^{2j-1} c^j_k
  f_k = \sum_{k=j}^{2j-1} c^j_k  (\Delta_\nu -I)f_k \\
&= \sum_{k=j}^{2j-1} c^j_k \left[ -2(\nu-k)f_{k+1} + k(k+1-2\nu)f_{k+2}\right]
\end{aligned}
\end{equation*}
by Lemma \ref{diff}. Rearranging, this equals
\begin{equation*}
\begin{aligned}
\sum_{k=j}^{2j-1} & 2(k - \nu)c^j_k f_{k+1} + \sum_{k=j+1}^{2j} (k-1)(k-2\nu)c^j_{k-1}f_{k+1}\\
=& \; \; 2(j- \nu) c^j_j f_{j+1} + (2j-1)(2j-2\nu) c^j_{2j-1}f_{2j+1} + \\
&+ \sum_{k=j+1}^{2j-1} \{2(k - \nu)c^j_k +  (k-1)(k-2\nu)c^j_{k-1}\}f_{k+1}\\
=& \; \; 2(j-\nu)\{f_{j+1} +(2j-1) c^j_{2j-1}f_{2j+1}\} +\\
& + \sum_{k=j+1}^{2j-1} \{2(k -\nu)c^j_k + (k-1)(k-2\nu)c^j_{k-1}\}f_{k+1}.
\end{aligned}
\end{equation*}

\medskip
\noindent
Now since
$$ 2(j- \nu)\psi_{j+1}  = 2(j-\nu) \sum_{k=j+1}^{2j+1} c^{j+1}_k f_k(r) =  2(j-\nu) \sum_{k=j}^{2j} c^{j+1}_{k+1} f_{k+1}(r)    $$
we see that the coefficients of $f_{j+1}$ agree (as $c_j^j = 1$ for all $j$), and it is therefore enough to see that 
$$ c^{j+1}_{2j+1} = (2j-1) c^j_{2j-1}$$
and, for $ j+1 \leq k \leq 2j-1$ and arbitrary $\nu$
$$ 2(j- \nu) c^{j+1}_{k+1} = 2(k- \nu)c^j_k +  (k-1)(k-2\nu)c^j_{k-1}.$$
This latter splits as 
 $$c^{j+1}_{k+1} = c^j_k +(k-1)c^j_{k-1}$$ 
(coefficient of $\nu$) and 
$$2j c^{j+1}_{k+1} = 2kc^j_k + (k-1)kc^j_{k-1}$$
($\nu=0$). All of these are true identities by Lemma \ref{comb}.
\end{proof}

\medskip
\noindent
\begin{corollary}\label{calc}
For $j \geq 0$ we have 
\begin{equation}\label{deriv}
\psi_j'(r) = -r \psi_{j+1}(r).
\end{equation}
For $m \geq 1$, $\nu \geq 0$, and either $0 \leq j \leq \nu -m$ or $j > \nu$ we have
\begin{equation}\label{gfd}
(\Delta_\nu - I)^m \psi_j = 2^m (j+m-1 - \nu)(j+m-2 - \nu)\cdots (j- \nu)\psi_{j+m},
\end{equation}
while for $\nu -m < j \leq \nu$ we have
\begin{equation}\label{dfg}
(\Delta_\nu - I)^m \psi_j = 0.
\end{equation}


\medskip
\noindent
Finally, for $m \geq 1$, $\nu \geq 0$ and $\psi= \sum_{j=0}^\nu \alpha_j \psi_j$, 
\begin{equation}\label{lap}
\Delta_\nu^m \psi = \sum_{j=0}^{\nu-m} \alpha_j 2^m (j+m-1 - \nu)\cdots (j- \nu)\psi_{j+m}
-  \sum_{k=0}^{m-1} (-1)^{m-k} {m \choose k} \Delta_\nu^k \psi.
\end{equation}
\end{corollary}

\begin{proof}
Write the first identity of Proposition \ref{key} as
$$ \psi''_j + \frac{2 \nu}{r} \psi_j' - \psi_j  = 2(j -\nu)\psi_{j+1},$$
divide by $\nu$, and let $\nu \to \infty$ to obtain \eqref{deriv}.

\medskip
\noindent
Let $j \geq 0$ and $m \geq 1$. By repeated application of Proposition \ref{key} we obtain \eqref{gfd} and 
\eqref{dfg}. 

\medskip
\noindent
By the binomial theorem, 
$$(\Delta_\nu - I)^m \psi_j =  \sum_{k=0}^m (-1)^{m-k} {m \choose k} \Delta_\nu^k \psi_j,$$
so that for $j \leq \nu - m$, by \eqref{gfd},
\begin{equation}\label{rew}
\Delta_\nu^m \psi_j = 2^m (j+m-1 - \nu)\cdots (j- \nu)\psi_{j+m}
- \sum_{k=0}^{m-1}(-1)^{m-k} {m \choose k} \Delta_\nu^k \psi_j
\end{equation}
while for $\nu - m < j \leq \nu$, by \eqref{dfg}, 
\begin{equation}\label{wer}
\Delta_\nu^m \psi_j = - \sum_{k=0}^{m-1}(-1)^{m-k} {m \choose k} \Delta_\nu^k \psi_j.
\end{equation}

\medskip
\noindent
Finally \eqref{lap} follows from \eqref{rew} and \eqref{wer} by taking linear combinations.
\end{proof}

\medskip
\noindent
We shall use Corollary \ref{calc} to systematically turn the calculus into algebra when fitting the appropriate 
boundary conditions to the problem, and when calculating the right hand side of formula \eqref{mainformula} 
in Theorem \ref{main} in the radial case. Note that \eqref{lap} tells us that for all $m \geq 1$, 
if $\psi= \sum_{j=0}^\nu \alpha_j \psi_j$, then $\Delta^m_\nu \psi$ can be expressed in terms of 
$\{\psi_0, \dots, \psi_\nu\}$.

\medskip
\noindent
Returning now to the ordinary differential equation \eqref{fds}, we see that
\eqref{dfg} from Corollary \ref{calc} tells us that the general solution to
$$(I - \Delta_\nu)^m h = 0$$ 
on $\{r> 0\} $ with $h(|\cdot|) \in L^2(\mathbb{R}^n)$ at infinity (for suitable $n$) 
is given by  a linear combination of $\{ \psi_\nu, \psi_{\nu-1}, \dots , \psi_{\nu-m +1}\}.$ 
(Routine calculations similar to those carried out at the beginning of
this section reveal that the other $m$ linearly independent solutions will feature
$e^{+r}$, and so will not have the requisite decay at infinity to belong to any
$L^2(\mathbb{R}^n)$.)




\medskip
\noindent
In particular, taking $\nu = (n-1)/2$ for $n$ odd, the general radial solution to
$$(I - \Delta)^m h = 0 \mbox{ on } \mathbb{R}^n \setminus \{0\}$$
with $h \in L^2(\mathbb{R}^n)$ at infinity is 
$$ \alpha_\nu \psi_v(|x|) + \cdots + \alpha_{\nu-m +1} \psi_{\nu-m +1}(|x|)$$
where the $\alpha$'s are arbitrary scalars. 

\medskip
\noindent
Specialising further, taking $m = (n+1)/2$,  we obtain:

\begin{proposition}\label{sox}
Suppose $n$ is odd. Then the general radial solution to
$$(I - \Delta)^{(n+1)/2} h = 0 \mbox{ on } \mathbb{R}^n \setminus \{0\}$$
with $h \in L^2(\mathbb{R}^n)$ at infinity is 
$$ h(x) = \alpha_0 \psi_0(|x|) + \cdots +  \alpha_{(n-1)/2} \psi_{(n-1)/2}(|x|)$$
where the $\alpha$'s are arbitrary scalars. 
\end{proposition}

\section{The algorithm for extremal energy and magnitude}\label{sec:alg}
\subsection{Fitting the boundary conditions}\label{AA}
The next task is to determine the $\alpha_j$'s which permit the solution 
$h = \sum_{j=0}^{(n-1)/2} \alpha_j \psi_j(|\cdot|)$ of
\begin{eqnarray*}
\begin{aligned}
(I - \Delta)^{(n+1)/2} h &= 0 \mbox{ on } \{|x| > R\} \\
h &= 1 \mbox{ on } \{|x| \leq R\}
\end{aligned}
\end{eqnarray*}
given by Proposition \ref{sox} to lie in $H^{(n+1)/2}(\mathbb{R}^n)$. This amounts to forcing it 
to satisfy appropriate boundary conditions
at $|x| = R$. 
Now our radial solution $h$ belongs to $H^{(n+1)/2}(\mathbb{R}^n)$ and so, by Lemma \ref{trace}, its derivatives 
of order $j$ lie in $H^{n/2 -j}$ of spheres (for $0 \leq j \leq (n-1)/2$) and moreover vary continuously 
with small changes in the radius of the sphere. Alternatively, the function $h$ defined on $(0, \infty)$ belongs
in $H^{(n+1)/2}(\mathbb{R})$ near $|x| = R$, and so by Sobolev embedding in the one-dimensional case, the 
derivatives of $h$ of order up to and including $(n-1)/2$ are continuous at $R$. In any case, since 
$h = 1 \mbox{ on } 
\{|x| \leq R\}$, 
this manifestly forces the $(n+1)/2$ conditions
\begin{eqnarray}\label{hgf}
\begin{aligned}
h(R) &= 1 \\
h'(R) &= 0 \\
\Delta h(R) &= 0\\
(\Delta h)'(R) &= 0\\
\vdots \\
(\Delta^{(n-1)/4} h)(R) &= 0 \mbox{ or }   (\Delta^{(n-3)/4} h)'(R) = 0
\end{aligned}
\end{eqnarray}
depending on whether $n-1$ or $n-3$ is a multiple of $4$. In either case, this gives us $(n+1)/2$
linear conditions on the $(n+1)/2$ unknowns $\alpha_0, \dots , \alpha_{(n-1)/2}$ which therefore 
determine them. (Since we know that there is a unique solution to the boundary value problem
we know that the linear system \eqref{hgf} has a unique solution. Hence the system is guaranteed to be 
nonsingular -- otherwise there would be either no solution or multiple solutions.)
In practice, we shall need to use Corollary \ref{calc} in conjunction with the 
explicit formulae for the $\psi_j$ in order to evaluate the $\alpha_j$. As above, we write 
$\nu = (n-1)/2$.

\medskip
\noindent
We shall use \eqref{lap} from Corollary \ref{calc}, directly on the equations involving $\Delta ^j h(R)$, 
and \eqref{lap} together with \eqref{deriv} on the equations involving $(\Delta ^j h)'(R)$.

\medskip
\noindent
The first equation is simply 
\begin{eqnarray*}
\alpha_0 \psi_0(R) + \cdots + \alpha_\nu \psi_\nu(R) = 1
\end{eqnarray*}
while subsequently we have, by \eqref{lap} and the boundary conditions on the terms
$h, \, \Delta h, \, \dots (\Delta)^{m-1}h$, 
$$\Delta ^m h(R) = 2^m \sum_{j=0}^{\nu -m}(j+m-1 - \nu)(j+m-2 - \nu)\cdots (j- \nu) \alpha_j \psi_{j+m}(R) - (-1)^{m}.$$
Since the terms $(j+m-r - \nu)$ are all negative and there are $m$ of them, the boundary condition
$\Delta ^m h(R) = 0$ reduces this to
$$ 2^m \sum_{j=0}^{\nu -m}(\nu - j)(\nu -j -1) \cdots (\nu -j -m +1)\alpha_j \psi_{j+m}(R) = 1.$$

\medskip
\noindent
Similarly, the boundary condition $(\Delta ^m h)'(R) = 0$ reduces to 
$$ 2^m \sum_{j=0}^{\nu -m}(\nu - j)(\nu -j -1) \cdots (\nu -j -m +1)\alpha_j \psi_{j+m+1}(R) = 0.$$

\medskip
\noindent
Thus we have the $\nu + 1 = (n+1)/2$ equations

\begin{eqnarray}\label{lincondns}
\begin{aligned}
\alpha_0 \psi_0(R) &+ \alpha_1 \psi_1(R) &+\cdots &+ \alpha_\nu \psi_\nu(R) &=& &1&\\
\alpha_0 \psi_1(R) &+ \alpha_1 \psi_2(R) &+\cdots &+ \alpha_\nu \psi_{\nu +1}(R) &=& &0&\\
2 \nu \alpha_0 \psi_1(R) &+ 2(\nu -1) \alpha_1 \psi_2(R) &+ \cdots &+ 2\alpha_{\nu -1} \psi_{\nu}(R) &=& &1&\\
2 \nu \alpha_0 \psi_2(R) &+ 2(\nu -1) \alpha_1 \psi_3(R) &+ \cdots &+ 2\alpha_{\nu-1} \psi_{\nu +1}(R) &=& &0&\\
4\nu (\nu -1) \alpha_0 \psi_2(R) &+ 4(\nu -1)(\nu -2) \alpha_1 \psi_3(R) &+ \cdots &+ 4.2.1\alpha_{\nu -2} \psi_{\nu}(R) &=& &1&\\
4\nu (\nu -1) \alpha_0 \psi_3(R) &+ 4(\nu -1)(\nu -2) \alpha_1 \psi_4(R) &+ \cdots &+ 4.2.1\alpha_{\nu -2} \psi_{\nu +1}(R) &=& &0&\\
&\vdots  &\vdots &  &=& &\vdots&\\
\end{aligned}
\end{eqnarray}
\medskip
\noindent
to solve for the $\nu +1 = (n+1)/2$ unknowns $\alpha_0, \, \alpha_1, \dots , \alpha_\nu$. (We know the system 
is nonsingular, but it is perhaps not so obvious that this is the case by direct inspection.) 
Notice that the ``input'' here involves only the values of $\psi_0(R), \dots, \psi_{\nu +1}(R)$.

\medskip
\noindent
We implement the explicit solution of these equations in 
the first few cases corresponding to $n=1, 3, 5$ and $7$ below. 

\subsection{Calculating the extremal energy}\label{BB}
Once we have determined the $\alpha_j = \alpha_j(R)$ which satisfy \eqref{lincondns}, we can use Theorem 
\ref{main} to calculate the extremal energy 
$\|h\|^2_{H^{(n+1)/2}(\mathbb{R}^n)}$ for $ h = \alpha_0 \psi_0 + \cdots + \alpha_\nu \psi_\nu$. Indeed, we have, with
$m = (n+1)/2$,
\begin{eqnarray}\label{mnb}
\begin{aligned}
 \|h\|^2_{H^{(n+1)/2}(\mathbb{R}^n)} &= {\rm Vol}(K) + \sum_{(n+1)/4 < j \leq (n+1)/2} (-1)^{j} {{\frac{n+1}{2}} \choose j}
\lim_{r \downarrow 0} \int_{\partial K_r} \frac{\partial}{\partial \nu} \Delta^{j-1} h {\rm d}S \\
&= \omega_n R^n + \sigma_{n-1}R^{n-1} \sum_{(n+1)/4 < j \leq (n+1)/2}(-1)^{j} {{\frac{n+1}{2}} \choose j}(\Delta^{j-1} h)'(R)\\
&= \omega_n R^n \left( 1 + nR^{-1}\sum_{(n+1)/4 < j \leq (n+1)/2} (-1)^{j} {{\frac{n+1}{2}} \choose j}(\Delta^{j-1} h)'(R)\right) 
\end{aligned}
\end{eqnarray}
where $\omega_n$ is the volume of the unit ball in $\mathbb{R}^n$ and $\sigma_{n-1} = n \omega_n$ 
is the surface area of the unit sphere $\mathbb{S}^{n-1}$ in
$\mathbb{R}^n$. We emphasise that by $(\Delta^{j-1} h)'(R)$ we mean
$\lim_{r \downarrow R}(\Delta^{j-1} h)'(r)$ here and subsequently
below. 

\medskip
\noindent
The calculation of $(\Delta^{j-1} h)'(R)$ is carried out using \eqref{lap} from Corollary \ref{calc}.
Indeed, this gives us, for $j \geq 2$,
\begin{equation*}
\begin{aligned}
\Delta^{j-1} h = & \; 2^{j-1} \sum_{i=0}^{\nu -j +1}(i+j-2 - \nu)(i+j-3 - \nu)\cdots (i- \nu) \alpha_i \psi_{i+j-1} \\
 &- \sum_{k=0}^{j-2} (-1)^{j-1-k} {{j-1} \choose k} \Delta^k h.
\end{aligned}
\end{equation*}
Now invoking $ \psi_j'(r) = -r \psi_{j+1}(r)$ (that is, \eqref{deriv} from Corollary \ref{calc}),
we obtain
\begin{equation*}
\begin{aligned}
(\Delta^{j-1} h)'(R) = &-R 2^{j-1} \sum_{i=0}^{\nu -j +1}(i+j-2 -
  \nu)(i+j-3 - \nu)\cdots (i- \nu) \alpha_i \psi_{i+j}(R) \\
&- \sum_{k=0}^{j-2} (-1)^{j-1-k} {{j-1} \choose k} (\Delta^k h)'(R).
\end{aligned}
\end{equation*}
The boundary conditions imply that we have $ (\Delta^k h)'(R) = 0$
for $0 \leq k \leq (\nu-2)/2$ when $\nu$ is even, and $0 \leq k \leq (\nu-1)/2$ when $\nu$ is odd. 
So in the $k$-sum we only need consider terms $k$ with $k> (\nu-1)/2$, and therefore, recalling that 
$\nu = (n-1)/2$, we have, for $j \geq 2$, 
\begin{equation}\label{iuy}
\begin{aligned}
(\Delta^{j-1} h)'(R) = -&R 2^{j-1} \sum_{i=0}^{\nu -j +1}(i+j-2 - \nu)\cdots (i- \nu) \alpha_i \psi_{i+j}(R) \\
-& \sum_{(\nu-1)/2 < k \leq j-2}(-1)^{j-1-k} {{j-1} \choose k} (\Delta^k h)'(R).
\end{aligned} 
\end{equation}
This allows us to effectively compute the terms $(\Delta^{j-1} h)'(R)$ recursively for $j \geq 2$ in terms of 
the values of $\psi_0(R), \dots, \psi_{\nu +1}(R)$ and $\alpha_0, \dots, \alpha_\nu$. We need this for those 
$j$ such that ${(n+1)/4 < j \leq (n+1)/2}$. Once done, we can substitute into formula 
\eqref{mnb} for the extremal energy, and we carry this out explicitly in the cases $n=1$, $3$, $5$ and $7$ below. 

\subsection{Summary of the algorithm}
We first solve the equations \eqref{lincondns} for the $(n+1)/2$ unknowns $\alpha_0, \dots, \alpha_{(n-1)/2}$.
Then we use formula \eqref{iuy} to calculate $(\Delta^{j-1} h)'(R)$ for ${(n+1)/4 < j \leq (n+1)/2}$, which 
we can substitute into formula \eqref{mnb}, whose value we then compute to give the extremal energy. 
Finally we divide by $\omega_n n!$ to obtain the magnitude.

\subsection{Proof of Theorem \ref{rat}}\label{sec:pf}
\medskip
\noindent
The discussion of the algorithm above immediately gives that 
the magnitude of the closed ball of radius $R$ in an odd-dimensional Euclidean space is a rational 
function of $R$. Indeed, modulo the multiplicative 
exponential term $e^{- R}$, all of the inputs $\psi_0(R), \dots ,\psi_\nu(R)$ are polynomials in $R^{-1}$ with 
nonnegative integer coefficients. The solutions $\alpha_0, \dots, \alpha_\nu$ to our system of equations 
will be given, by Cram\'er's rule, by a ratio of two determinants, each of which is a polynomial 
in $\psi_0(R), \dots ,\psi_\nu(R)$ with integer coefficients, up to a multiplicative exponential term $e^R$. 
Inspection of formulae \eqref{iuy} and \eqref{mnb} indicates that the magnitude is given by $R^n/n!$ plus 
an integer combination of various $\alpha_r \psi_s$ for $0 \leq r \leq (n-1)/2$ and $0 \leq s \leq (n+1)/2$. 
So the multiplicative exponential terms $e^{\pm R}$ will cancel throughout in every instance, and we will
be left with $R^n/n!$ plus a rational function of $R$ with integer coefficients.

\section{Implementation of the algorithm}\label{sec:imp}

\medskip
\noindent
Recall that we have
$$
\begin{array}{lllll}
\psi_0(r)=e^{-r} \vspace{0.25cm} \\  
\psi_1(r)=\frac{e^{-r}}{r} \vspace{0.25cm} \\ 
\psi_2(r)=e^{-r}\left( \frac{1}{r^2} + \frac{1}{r^3}    \right) \vspace{0.25cm} \\   
\psi_3(r)=e^{-r}\left( \frac{1}{r^3} + \frac{3}{r^4}+\frac{3}{r^5}    \right) \vspace{0.25cm} \\  
\psi_4(r)=e^{-r}\left( \frac{1}{r^4} + \frac{6}{r^5}+\frac{15}{r^6}+\frac{15}{r^7}   \right).
\vspace{0.25cm}
\end{array}$$

\subsection{Case $n=1$.}

\medskip
\noindent
{\em Notational fix:} $\nu = 0$, $m=1$.

\medskip
\noindent
{\em Inputs:} $\psi_0(R)$, $\psi_1(R)$.

\medskip
\noindent
{\em Outputs:} $ \alpha_0$, $h'(R)$, extremal energy, magnitude.

\medskip
\noindent
The single equation for $\alpha_0$ is $\alpha_0 \psi_0(R) = 1$ which has solution $\alpha_0 = 1/\psi_0(R)$.

\medskip
\noindent
Equation \eqref{mnb} requires $(\Delta^{j-1} h)'(R)$ for $j=1$ only, which is simply $h'(R)$, and
by Corollary \ref{calc}, equation \eqref{deriv}, $h'(R) = - R \alpha_0\psi_1(R)$.

\medskip
\noindent
Using \eqref{mnb}, the extremal energy is given by
$$\omega_1 R\left(1 +  R^{-1} (-1) {1 \choose 1} h'(R)\right)= \omega_1 R \left(1  - R^{-1}h'(R) \right)
= \omega_1 R\left(1 +  \alpha_0\psi_1(R)\right).$$
Substituting the value of $\alpha_0$ gives  
$$\omega_1 R\left(1 + \psi_1(R)/\psi_0(R)\right) = \omega_1(R +1).$$ 
Dividing by $1! \omega_1$ gives the magnitude  of the closed interval $[-R, R]$ as 
$$\boxed{R + 1,}$$ 
in agreement with the results of \cite{Lein13} and \cite{Meckes14}.

\subsection{Case $n=3$.}

$\,$

\vspace{0.1cm}

\noindent
{\em Notational fix:} $\nu = 1$, $m=2$.

\medskip
\noindent
{\em Inputs:} $\psi_0(R)$, $\psi_1(R)$, $\psi_2(R)$.

\medskip
\noindent
{\em Outputs:} $ \alpha_0$, $\alpha_1$, $h'(R)$, $(\Delta h)'(R)$, extremal energy, magnitude.

\medskip
\noindent
Our two equations are
$$ \alpha_0 \psi_0(R) + \alpha_1 \psi_1(R) = 1$$
$$ \alpha_0 \psi_1(R) + \alpha_1 \psi_2(R) = 0$$
which have solution
$$ \alpha_0 = \frac{\psi_2(R)}{\psi_0(R) \psi_2(R) - \psi_1(R)^2}$$
$$ \alpha_1 = \frac{-\psi_1(R)}{\psi_0(R) \psi_2(R) - \psi_1(R)^2}.$$
Now 
$\psi_0 (R)\psi_2(R) - \psi_1(R)^2 = e^{-2R}(R^{-2} + R^{-3} - R^{-2}) =  e^{-2R}R^{-3}$,
so 
\begin{eqnarray*}
\begin{aligned}
\alpha_0 &= e^{R}(R+1) \\ 
\alpha_1 &= - e^R R^2.
\end{aligned}
\end{eqnarray*}

\medskip
\noindent
Note the appearance of the expression $R+1$ (as displayed in the previous subsection) in the formula for $\alpha_0$.

\medskip
\noindent
Equation \eqref{mnb} requires $(\Delta^{j-1} h)'(R)$ for $j=2$ only, and by formula \eqref{iuy}
we have
$$(\Delta h)'(R) = -R(2)^1(-1)\alpha_0 \psi_2(R) - 0 = 2R\alpha_0 \psi_2(R)$$

\medskip
\noindent
Using \eqref{mnb}, the extremal energy is given by
$$ \omega_3R^3\left(1 + 3R^{-1}(-1)^2 {2 \choose 2}(\Delta h)'(R)\right) = \omega_3R^3\left(1 +3R^{-1}(\Delta h)'(R)\right)$$
$$ = \omega_3R^3\left(1 + 6\alpha_0 \psi_2(R)\right).$$

\medskip
\noindent
If we substitute in the explicit values for $\alpha_0$ and $\psi_2(R)$ we obtain
$$ = \omega_3(R^3 + 6R^2 + 12 R +6), $$
which upon dividing by $3! \omega_3$ gives the magnitude of the closed ball of radius $R$ in $\mathbb{R}^3$ as
$$ \boxed{\frac{1}{3!} (R^3 + 6R^2 + 12R +6) = \frac{R^3}{3!} + R^2 + 2R +1.}$$
This verifies the Leinster--Willerton convex magnitude conjecture for balls in $\mathbb{R}^3$.

\subsection{Case $n=5$.}
$ \, $

\vspace{0.1cm}

\noindent
{\em Notational fix:} $\nu = 2$, $m=3$.

\medskip
\noindent
{\em Inputs:} $\psi_0(R), \dots , \psi_3(R)$.

\medskip
\noindent
{\em Outputs:} $ \alpha_0$, $\alpha_1$, $\alpha_2$, $h'(R)$, $\Delta h'(R)$,  $(\Delta^2 h)'(R)$, extremal energy, magnitude.

\medskip
\noindent
Our three equations are
\begin{eqnarray*}
\begin{aligned}
\alpha_0 \psi_0(R) &+ \alpha_1 \psi_1(R) &+&  \alpha_2 \psi_2(R) &=& &1& \\
\alpha_0 \psi_1(R) &+ \alpha_1 \psi_2(R) &+& \alpha_2 \psi_3(R) &=& &0&\\
4 \alpha_0 \psi_1(R) &+ 2  \alpha_1  \psi_2(R) &+& 0 &=& &1&
\end{aligned}
\end{eqnarray*}
which upon substitution of the values of $\psi_j(R)$ become
\begin{eqnarray*}
\begin{aligned}
\alpha_0 &+ \alpha_1 R^{-1} &+& \alpha_2 (R^{-2} + R^{-3}) &=& &e^R& \\
\alpha_0 R^{-1} &+  \alpha_1(R^{-2} + R^{-3}) &+& \alpha_2(R^{-3} + 3 R^{-4} + 3 R^{-5}) &=& &0& \\
\alpha_0 4R^{-1} &+ \alpha_1 2(R^{-2} + R^{-3}) &+&0 &=& &e^R.&
\end{aligned}
\end{eqnarray*}

\medskip
\noindent
A visit to \url{matrixcalc.org} confirms that the solution of these equations is given by 
$$\alpha_j = \frac{e^R \beta_j}{2(R+3)}$$
where 
\begin{eqnarray*}
\begin{aligned}
\beta_0 &= 6 + 12 R + 6R^2 + R^3\\
\beta_1 &= -R^2(12 + 9R + 2R^2)\\
\beta_2 &= (2+R)R^4.\\
\end{aligned}
\end{eqnarray*}
Note that the system is singular when $R = -3$, and note the appearance of the expression $R^3/6 + R^2 + 2 R + 1$ (as displayed in the previous subsection)
in the formula for $\alpha_0$.

\medskip
\noindent
Equation \eqref{mnb} requires $(\Delta^{j-1} h)'(R)$ for $j=2$ and $3$, and by formula \eqref{iuy}
we have (case $j = 2$)
$$ (\Delta h)'(R) = -R.2 \sum_{i=0}^{1}(i - 2) \alpha_i \psi_{i+2}(R) + 0 = 
2R(2 \alpha_0 \psi_2(R) + \alpha_1 \psi_3(R)),$$
and, (case $j=3$),
\begin{equation*}
\begin{aligned}
(\Delta^2 h)'(R) &= -R 2^{2} (-1)(-2)\alpha_0 \psi_{3}(R) - (-1)^{1}
  {2 \choose 1} (\Delta h)'(R) \\
 &= -8R \alpha_0 \psi_3(R) + 4R(2 \alpha_0 \psi_2(R) + \alpha_1 \psi_3(R))\\
 &= -8R \alpha_0 \psi_3(R) + 8R  \alpha_0 \psi_2(R) + 4R  \alpha_1 \psi_3(R).
\end{aligned}
\end{equation*}
\medskip
\noindent
Using \eqref{mnb}, the extremal energy is given by
\begin{eqnarray*}
\begin{aligned}
\omega_5 & R^5\left(1  + 5 R^{-1}\Big\{ {3 \choose 2} (\Delta h)'(R) - {3 \choose 3}  (\Delta^2 h)'(R)\Big\}\right) \\
& = \omega_5 R^5\left(1 + 5R^{-1}\{3 (\Delta h)'(R) -  (\Delta^2 h)'(R)\}\right)\\
& = \omega_5 R^5\left(1 + 5\{3 [4 \alpha_0  \psi_2(R) + 2 \alpha_1  \psi_3(R)] 
-[-8\alpha_0 \psi_3(R) + 8  \alpha_0 \psi_2(R) + 4 \alpha_1 \psi_3(R)]\}\right)\\
& = \omega_5 R^5 \left(1 + 5\{ 4\alpha_0 \psi_2(R) + 8 \alpha_0 \psi_3(R) + 2 \alpha_1 \psi_3(R)\}\right). \\
\end{aligned}
\end{eqnarray*}

\medskip
\noindent
Substituting in the explicit values for $\psi_2(R)$ and
$\psi_3(R)$ we see that this expression equals
\begin{eqnarray*}
\begin{aligned}
& \; \; \omega_5 \left(R^5+ 5e^{-R}R^5[(4\alpha_0(R^{-2} + R^{-3}) + 8\alpha_0(R^{-3} + 3R^{-4} + 3R^{-5}) 
+ 2 \alpha_1 (R^{-3} + 3R^{-4} + 3R^{-5})]\right) \\
= & \; \; \omega_5 \left(R^5 + 5e^{-R}[ \alpha_0(4R^3 + 12R^2 + 24R + 24) + \alpha_1 (2R^2 + 6R +6)]\right).
\end{aligned}
\end{eqnarray*}

\medskip
\noindent
Now plugging in the explicit values for $\alpha_0$ and $\alpha_1$, we obtain
$$ \omega_5 R^5 + \frac{5 \omega_5}{2(R+3)} \left((6 + 12 R + 6R^2 + R^3)(4R^3 + 12R^2 + 24R + 24) - 
R^2(12 + 9R + 2R^2)(2R^2 + 6R +6)\right)$$
which, after some simplification and cancellation in the highest order terms, 
gives that the extremal energy is given by
$$ \omega_5 R^5 + 5 \omega_5 \frac{72 + 216 R + 216 R^2 + 105 R^3 + 27 R^4 + 3 R^5}{R +3}.$$
Dividing by $c_5 = 5! \omega_5$ we obtain that the 
magnitude of the ball of radius $R$ in $\mathbb{R}^5$ is 
$$ \boxed{\frac{R^5}{5!} +  \frac{3R^5 +  27 R^4 + 105 R^3  + 216 R^2  + 216 R + 72}{24(R +3)}.}$$
Note that when $R = 0$ we obtain the value $1$, but when $R>0$ this quantity is strictly greater than 
the conjectured value of
$$  \frac{1}{5!}R^5  + \frac{1}{9}R^4  + \frac{2}{3} R^3  + 2 R^2 + \frac{8}{3} R + 2 R^2 + 1.$$

\medskip
\noindent
We can expand our formula when $R> 3$ as
$$ \frac{1}{5!}R^5 + \frac{1}{8}R^4 + \frac{3}{4} R^3 + \frac{17}{8} R^2 +  \frac{21}{8} R + \frac{9}{8} +O(\frac{1}{R})$$
as $R \to \infty$, but the coefficients (other than that of $R^5$) do not agree either with those 
conjectured by Leinster and Willerton. 

\subsection{Case $n=7$.}
$ \; $

\vspace{0.1cm}

\noindent
{\em Notational fix:} $\nu = 3$, $m=4$.

\medskip
\noindent
{\em Inputs:} $\psi_0(R), \dots , \psi_4(R)$.

\medskip
\noindent
{\em Outputs:} $ \alpha_0, \dots ,\alpha_3$, $h'(R), \dots , (\Delta^3 h)'(R)$, extremal energy, magnitude.

\bigskip
\noindent 
For brevity of expression we suppress the multiplicative factors $e^{\pm R}$ which should appear below, and which 
ultimately cancel as described above. We have the four equations
\begin{equation}\label{sistema1}
\left\{   \begin{array}{llll}
\alpha_0 \psi_0(R)+ \alpha_1 \psi_1(R)+ \alpha_2 \psi_2(R)+ \alpha_3 \psi_3(R)=1 \vspace{0.25cm} \\
\alpha_0 \psi_1(R)+ \alpha_1 \psi_2(R)+ \alpha_2 \psi_3(R)+ \alpha_3 \psi_4(R) =0  \vspace{0.25cm} \\
6 \alpha_0 \psi_1(R)+4 \alpha_1 \psi_2(R)+2 \alpha_2 \psi_3(R)=1  \vspace{0.25cm} \\ 
3 \alpha_0 \psi_2(R)+2 \alpha_1 \psi_3(R)+ \alpha_2 \psi_4(R)=0
\end{array}     \right.
\end{equation}
with solution
\begin{equation}\label{sistema2}
\begin{array}{llll}
\alpha_0 = \frac{360+1080R+1080 R^2+ 525 R^3+135 R^4+18 R^5+R^6}{360+288 R+72R^2+6R^3}, \vspace{0.25cm} \\ 
\alpha_1=-\frac{360R^2+555R^3+345R^4+105R^5+16R^6+R^7}{120+96R+24R^2+2R^3}, \vspace{0.25cm} \\ 
\alpha_2=\frac{120R^4+150R^5+66R^6+13R^7+R^8}{120+96R+24R^2+2R^3}, \vspace{0.25cm} \\ 
\alpha_3=- \frac{24R^6+27R^7+9R^8+R^9}{360+288R+72R^2+6R^3}
\end{array}
\end{equation}
as another visit to \url{matrixcalc.org} confirms.

\medskip
\noindent
Note that the numerator in the formula for $\alpha_0$ is 
$$ 360+1080R+1080 R^2+ 525 R^3+135 R^4+18 R^5+R^6$$ 
whereas the magnitude of the ball of radius $R$ in $\mathbb{R}^5$ (calculated in the 
previous subsection) is
\begin{equation*}
\begin{aligned}
& \; \; \frac{R^5}{5!} +  \frac{3R^5 +  27 R^4 + 105 R^3  + 216 R^2  + 216 R
  + 72}{24(R +3)} \\
= & \; \; \frac{360+1080R+1080 R^2+ 525 R^3+135 R^4+18 R^5+R^6}{120(R+3)},
\end{aligned}
\end{equation*}
with the same numerator as $\alpha_0$.

\medskip
\noindent
From equation \eqref{mnb} the extremal energy is given by
\begin{eqnarray}\label{hb}
\begin{aligned}
&\omega_7 R^7 + 7\omega_7 R^6\left(-{4 \choose 3}\left( \Delta^2 u
\right)'(R) + {4 \choose 4}\left( \Delta^3 u \right)'(R)\right)\\
=& \; \omega_7\left( R^7 + 7R^6(-4\left( \Delta^2 u\right)'(R) + \left( \Delta^3 u \right)'(R)\right)\\ 
=& \; \omega_7\left( R^7 -28R^6\left( \Delta^2 u\right)'(R) + 7R^6\left( \Delta^3 u \right)'(R)\right)\\
=& \; \omega_7\left( R^7 + 7R^6\left\{ \left(\Delta^3 u \right)'(R) - 4\left( \Delta^2 u\right)'(R) \right\}\right).\\
\end{aligned}
\end{eqnarray} 

\medskip
\noindent
So we need to calculate $ \left( \Delta^2 u \right)'(R)$ and $ \left(\Delta^3 u \right)'(R)$. 


\medskip
\noindent
Using the recurrence relation \eqref{iuy} with $j=3$ gives
\begin{equation*}
\begin{aligned}
(\Delta^{2} h)'(R) &= -R 2^{2} \sum_{i=0}^{1}(i-2)(i - 3)\alpha_i
  \psi_{i+3}(R) +0 \\
 &= -4R(6 \alpha_0 \psi_{3}(R) + 2\alpha_1 \psi_{4}(R)) \\
&= -8R(3\alpha_0 \psi_{3}(R) + \alpha_1 \psi_{4}(R)).
\end{aligned}
\end{equation*}
\medskip
\noindent
Next, using the recurrence relation \eqref{iuy} with $j=4$ gives
\begin{equation*}
\begin{aligned}
(\Delta^{3} h)'(R) 
&= 6R 2^{3} \alpha_0 \psi_{4}(R) + {{3} \choose 2}  (\Delta^2 h)'(R)
  \\
& = 48R \alpha_0 \psi_{4}(R) +3(-8R(3\alpha_0 \psi_{3}(R) + \alpha_1 \psi_{4}(R)))\\
&=  48R \alpha_0 \psi_{4}(R) -24 R(3\alpha_0 \psi_{3}(R) + \alpha_1
  \psi_{4}(R)) \\
&= 24R(2\alpha_0 \psi_{4}(R) -3\alpha_0 \psi_{3}(R) -\alpha_1 \psi_{4}(R)).
\end{aligned}
\end{equation*}

\vspace{1cm}
\noindent
Now
\begin{equation*}
\begin{aligned}
3\alpha_0\psi_3+\alpha_1\psi_4&=\frac{360+1080R+1080R^2+525R^3+135R^4+18R^5+R^6}{120+96R+24R^2+2R^3} \times \frac{3+3R+R^2}{R^5} \\
& \; \; \; \; -\frac{360+555R+345R^2+105R^3+16R^4+R^5}{120+96R+24R^2+2R^3}\times \frac{15+15R+6R^2+R^3}{R^5}\\
&=-\frac{4320+9405R+8820R^2+4545R^3+1380R^4+246R^5+24R^6+R^7}{R^5 \left(  120+96R+24R^2+2R^3 \right)}
\end{aligned}
\end{equation*}
(noticing the cancellation in the $R^8$ terms), and so
\begin{equation*}\label{antes11}
\left( \Delta^2 u \right)'(R)=\frac{8(4320+9405R+8820R^2+4545R^3+1380R^4+246R^5+24R^6+R^7)}
{R^4 \left(  120+96R+24R^2+2R^3 \right)}.
\end{equation*}

\vspace{1cm}
\noindent
Moreover
$$3\psi_3-2\psi_4=\frac{3}{R^3}+\frac{9}{R^4}+\frac{9}{R^5}-\frac{2}{R^4}-\frac{12}{R^5}-\frac{30}{R^6}-\frac{30}{R^7}=\frac{-30-30R-3R^2+7R^3+3R^4}{R^7},$$
so that
\begin{equation*}
\alpha_0 \left(  3 \psi_3-2 \psi_4  \right)+\alpha_1 \psi_4 
\end{equation*}
equals
\begin{equation*}
\begin{aligned}
&\;\;\;\; \frac{360+1080R+1080R^2+525R^3+135R^4+18R^5+R^6}{360+288R+72R^2+6R^3} \times \frac{-30-30R-3R^2+7R^3+3R^4}{R^7} \\ 
& \; \; \; \; -\frac{1080R^2+1665R^3+1035R^4+315R^5+48R^6+3R^7}{360+288R+72R^2+6R^3} \times \frac{15+15R+6R^2+R^3}{R^7}\\
&=-\frac{10800+43200R+82080R^2+90045R^3+61380R^4+26685R^5+7380R^6+1254R^7+120R^8+5R^9}{R^7 \left(  360+288R+72R^2+6R^3 \right)}\\
\end{aligned}
\end{equation*}
(noticing the cancellation in the $R^{10}$ terms), which gives that
\begin{equation*}\label{antes22}
\left( \Delta^3 u \right)'(R)
\end{equation*}
equals
$$\frac{24(10800+43200R+82080R^2+90045R^3+61380R^4+26685R^5+7380R^6+1254R^7+120R^8+5R^9)}
{R^6 \left(  360+288R+72R^2+6R^3 \right)}.$$

\vspace{1cm}
\noindent
Combining these, we see that
$$R^6\left(\Delta^3u\right)'(R)-4R^6\left(\Delta^2u\right)'(R)$$ 
equals
\begin{equation*}
\begin{aligned}& \;\;\;\; \frac{24(10800+43200R+82080R^2+90045R^3+61380R^4+26685R^5+7380R^6+1254R^7+120R^8+5R^9)}{
  360+288R+72R^2+6R^3 } \\
& \;\;\;\; - \frac{96(4320R^2+9405R^3+8820R^4+4545R^5+1380R^6+246R^7+24R^8+R^9)}{ 360+288R+72R^2+6R^3 } \\
&=\frac{43200+172800R+259200R^2+209700R^3+104400R^4+34020R^5+7440R^6+1080R^7+96R^8+4R^9}{60+48R+12R^2+R^3}.
\end{aligned}
\end{equation*}

\vspace{1cm}
\noindent
Finally, we insert this into \eqref{hb} and divide by $ 7!\omega_7$ to see that the magnitude of the 
closed ball of radius $R$ in $\mathbb{R}^7$ is given by
$$\boxed{\; \frac{R^7}{7!} + \frac{\frac{1}{180}R^9 +\frac{2}{15}R^8 +\frac{3}{2}R^7+\frac{31}{3}R^6
+\frac{189}{4}R^5+145R^4+\frac{1165}{4}R^3+360R^2+240R+60}{R^3+12R^2+48R+60} \;.}$$

\section{Concluding remarks}\label{sec_final}

\medskip
\noindent
1. The three-dimensional convex conjecture predicts that the magnitude of the convex body $K$ 
should be
$$ \frac{V_0(K)}{0! \; \omega_0} + \frac{V_1(K)}{1! \; \omega_1} + \frac{V_2(K)}{2! \; \omega_2} + \frac{V_3(K)}{3! \; \omega_3}$$
where $V_j(K)$ is the $j$'th intrinsic volume and $\omega_0 = 1$, $\omega_1 =2 $,  $\omega_2 = \pi$
and $\omega_3 = 4 \pi/3$. Now $V_0(K) = 1$,  $V_2(K) = \mbox{Surf }(\partial K)/2 $, 
$V_3(K) = \mbox{Vol}(K)$ and $V_1(K) = P_2(K)/ \pi$ where $P_2(K)$ is the coefficient of $t^2$ 
in the degree $3$ polynomial $\mbox{vol }(K + t B_1)$. So the magnitude of $K$ should be
$$ \frac{\mbox{Vol }(K)}{8 \pi} +\frac{\mbox{Surf }(\partial K)}{4 \pi} + \frac{V_1(K)}{2} +1 .$$
In the notation of Theorem \ref{main}, the question of validity of the convex magnitude conjecture 
reduces to whether for the solution 
$h \in H^2(\mathbb{R}^3)$ of the boundary value problem $(I - \Delta)^2 h = 0$ off $K$, $h=1$ on $K$, we have
$$ \frac{1}{8 \pi} \int_{\partial K}\frac{\partial \Delta h }{\partial \nu} {\rm d}S
= \frac{\mbox{surf }(\partial K)}{4 \pi}+ \frac{V_1(K)}{2}+ 1.$$

\medskip
\noindent
When $K$ is the cuboid $[-R_1, R_1] \times [-R_2, R_2] \times [-R_3, R_3]$ we have 
$V_1(K) = 2(R_1 + R_2 + R_3),$ so that we expect the magnitude of $K$ to be
$$  \frac{R_1 R_2 R_3}{ \pi} + \frac{(R_1 R_2 + R_2 R_3 + R_3 R_1)}{2\pi} + (R_1 + R_2 + R_3) + 1.$$
In particular in the case of the cuboid do we have
$$ \frac{1}{8 \pi} \int_{\partial K}\frac{\partial \Delta h }{\partial \nu} {\rm d}S 
= \frac{(R_1 R_2 + R_2 R_3 + R_3 R_1)}{2\pi}+ (R_1 + R_2 + R_3)+ 1?$$
Even when $R_j = R$ for $j = 1,2,3$, why should the quantity 
$$ \frac{1}{8 \pi} \int_{\partial K}\frac{\partial \Delta h }{\partial \nu} {\rm d}S$$ 
be a polynomial of degree $2$, still less the specific polynomial $3R^2/2\pi + 3R + 1$?  
One of the difficulties in addressing this problem lies in not having explicit formulae 
for solutions of these types of boundary value problems in the absence of spherical symmetry.
 
\medskip
\noindent
On the other hand, matters may be better if we instead consider ellipsoids. 
Khavinson and Lundberg \cite{KL} note that ellipsoidal harmonics can be used to 
study elliptic boundary value problems in the interior of ellipsoids. It may be possible to 
use them in the current context to obtain explicit solutions of 
$(I - \Delta)^2 u = 0$ in the exterior of an ellipsoid in $\mathbb{R}^3$ which 
might then lead to explicit formulae for the magnitudes of ellipsoids in $\mathbb{R}^3$.
Note however that the mixed volumes, even the surface area, of a general ellipsoid in $\mathbb{R}^3$
are not expressible as elementary functions of the semi-axes, but instead involve elliptic integrals. 

\medskip
\noindent
2. The methods presented here can be extended to obtain explicit expressions
for the magnitude of general compact sets possessing spherical symmetry.

\medskip 
\noindent
3. An as-yet-unexplained (empirical) phenomenon is that in the formula for the coefficient $\alpha_0(R)$ 
in the $n$-dimensional setting, with $n$ odd and small, the formula for magnitude of the ball of 
radius $R$ in $\mathbb{R}^{n-2}$ makes a rather mysterious appearance, see the previous section. 
Further observations of this kind have been made by Willerton, (private communication), who has also
developed a more streamlined approach to solving the ODE \eqref{fds} and fitting the boundary conditions.
This leads to a more symmetric system of linear equations for the
unknowns $\alpha_j$. Using the new approach to calculating magnitude of
Leinster and Meckes (\cite{LM16}, Theorem 4.16), he is able to derive a formula for the
magnitude of a ball in odd dimensions expressed in terms of the reverse 
Bessel polynomials. This formula can be re-cast in a reasonably succinct and natural form 
in odd dimensions at least up to and including $39$ (where the coefficients of the 
numerator in the magnitude total over $60,000$ digits).

\medskip
\noindent
4. The methods here can also be used to study the extremal energies or Bessel-like capacities 
$$ C_m(K, \lambda) := \inf \{\|(\lambda I - \Delta)^{m/2} h \|^2_{L^2(\mathbb{R}^n)} \; : \; h \equiv 1 \mbox{ on } K\}$$
for $\lambda > 0$ for convex bodies $K$ when $m \in \mathbb{N}$ and to calculate them explicitly when
$K$ is a ball. Indeed, a simple scaling argument shows that
$$ C_m(K, \lambda) = \lambda^{m-n/2}C_m(\lambda^{1/2}K,1)$$
and we have shown the asymptotic behaviour of $C_{(n+1)/2}(K,1)$ and how to calculate it explicitly for 
balls above. In particular, $C_{(n+1)/2}(K,\lambda) \sim \lambda^{1/2}$ as $\lambda \to 0$ and  $C_{(n+1)/2}(K,\lambda) \sim \lambda^{(n+1)/2} {\rm Vol }(K)/n! \omega_n$ as $\lambda \to \infty$. When $m=1$ and $n=3$ for example,
we can calculate using Theorem \ref{main} and the formula for $\psi_1$ that $C_1(B(0,R),1)$ is simply
$4 \pi(R^3/3 + R +1)$. So
\begin{equation*}
\begin{aligned}
C_1(B(0,R), \lambda) &=  \lambda^{-1/2}C_1(B(0,\lambda^{1/2}R),1) \\
&= 4 \pi \lambda^{-1/2}(\lambda^{3/2}R^3/3 + \lambda^{1/2}R +1)\\
&=4 \pi \lambda R^3/3 + 4 \pi R + 4 \pi \lambda^{-1/2}.
\end{aligned}
\end{equation*}

\medskip
\noindent
5. There remains the possibility that a reformulated version of the
convex magnitude conjecture might be
true if we replaced magnitude by some closely related notion which coincides with it
in dimensions one and three, and is strictly smaller than it for nontrivial balls in odd 
dimensions five and above. It would be very interesting to find a such a notion. The notion of maximum 
diversity 
\begin{equation*} 
|X|_+ := \frac{1}{n! \omega_n} \inf\left\{\|f\|^2_{H^{(n+1)/2}(\mathbb{R}^n)} \, : \, 
f \in H^{(n+1)/2}(\mathbb{R}^n), f \geq  1 \mbox{ on } X\right\}
\end{equation*}
found for example in \cite{Meckes13}, \cite{Meckes14} and the references therein, 
and which is much more in-tune than magnitude with classical 
potential theory, unfortunately does not satisfy these criteria. One should note 
the parallel between the definition of maximum diversity and certain obstacle problems for 
higher-order operators (see for example \cite{CF} and \cite{CFT}) though this does not seem to have been 
exploited yet.) We thank E. Milakis for bringing this parallel to our attention.

\medskip
\noindent
6. In forthcoming work, H. Gimperlein and M. Goffeng \cite{GG} use methods from 
semiclassical analysis together with our Theorem \ref{main} to exhibit an asymptotic expansion of
the magnitude $|R X|$ as $R \to \infty$, where $X \subseteq \mathbb{R}^{n}$ is open and bounded
with smooth boundary, when $n \geq 3$ is odd. In particular, they prove 
$$|R X|  = \frac{\mathrm{Vol}(X)}{n!\ \omega_n} R^{n} + \frac{(n+1)\ \mathrm{Surf}(\partial X)}{2\ n!\ \omega_n} R^{n-1} + c R^{n-2} + o(R^{n-2}) \ ,$$ 
where $c$ is proportional to an integral of a curvature of $\partial X$. (The coefficient of $R^{n-1}$ differs 
from the prediction of Leinster and Willerton unless $n=3$.) 

\bibliographystyle{plain}

\end{document}